\documentclass[a4paper,reqno,10pt]{amsart}

\usepackage{epsf,graphicx,epsfig}
\usepackage{amscd}
\usepackage{amsmath,latexsym,amssymb,amsthm}
\usepackage[nospace,noadjust]{cite}
\usepackage{textcomp}
\usepackage{dsfont}
\usepackage{fdsymbol}
\usepackage{setspace,cite}
\usepackage{lscape,fancyhdr,fancybox}
\usepackage{stmaryrd}
\usepackage[all,cmtip]{xy}
\usepackage{tikz}
\usepackage{cancel}
\usetikzlibrary{shapes,arrows,decorations.markings}
\usepackage{graphicx}

\usepackage{color}
\usepackage{url}
\usepackage{enumerate}
\usepackage[mathscr]{euscript}

  \theoremstyle{plain}

    \newtheorem{thm}{Theorem}[section]
    \newtheorem{prop}[thm]{Proposition}

    \newtheorem{subsec}[thm]{}
\theoremstyle{definition}
    \newtheorem{defn}[thm]{Definition}

\theoremstyle{remark}

\newcommand{\Hom}{\operatorname{Hom}}

\title{}
\author{}
\date{}
\usepackage{amssymb}

\usepackage{hyperref}
\hypersetup{
	colorlinks,
	citecolor=blue,
	filecolor=black,
	linkcolor=blue,
	urlcolor=black
}

\begin{document}
\allowdisplaybreaks

\title[Cohomology and automorphisms for a matched pair of
3-Lie algebras
]{Cohomology and automorphisms for a matched pair of
3-Lie algebras
}

\author{Tao Zhang}
\address{School of Mathematics and Statistics,\\
Henan Normal University, Xinxiang 453007, P. R. China;\\
 E-mail address:\texttt{{  zhangtao@htu.edu.cn}}}

\author{Jingzi Zhang}
\address{School of Mathematics and Statistics,\\
Henan Normal University, Xinxiang 453007, P. R. China;\\
 E-mail address:\texttt{{  zhangjingzi@163.com}}}

\begin{abstract}
We begin by reviewing the definition of 3-Lie algebras and the fundamental concepts of matched pairs. Subsequently, we introduce the representation theory of matched pairs and define the semidirect product. Building on this foundation, we define the low-dimensional cohomology groups of matched pairs. In addition, we explore the infinitesimal deformations and abelian extensions of matched pairs. Finally, we examine the inducibility of automorphisms of matched pairs and present related results through the Wells exact sequence.
\end{abstract}

\maketitle



\medskip





\thispagestyle{empty}

\tableofcontents


\medskip

\section{Introduction}\label{sec1}
 3-Lie algebras are a significant class of algebraic structures that have garnered considerable attention due to their applications in mathematical physics, algebraic geometry, and representation theory. The concept of 3-Lie algebras can be traced back to the work of Nambu \cite{nambu}. In 1973, Nambu proposed generalized Hamiltonian mechanics and introduced the ternary bracket, which laid the foundation for the definition of 3-Lie algebras. These algebras extend the concept of Lie algebras by incorporating a ternary bracket, which satisfies a generalized Jacobi identity. In recent years, research on the representation theory, cohomology and deformation theory of 3-Lie algebras has seen a surge. For example, \cite{zhang0} had made significant contributions to the theory of deformations and extensions of 3-Lie color algebras.  Bai \cite{bai} investigated the bialgebras of 3-Lie algebras and the classical Yang-Baxter equation, highlighting their connections with Manin triples. It worth also mentioning that \cite{xu} investigated the related problem of abelian extensions of 3-Lie algebras, and \cite{song} studies issued related to the non-abelian extensions of 3-Lie algebras.

 The notion of matched pairs, originally introduced by Majid \cite {majid1, majid2} in the context of Lie groups and later extended to Lie algebras, which used to study the set-theoretical solutions of the Yang-Baxter equation and provided a powerful framework for understanding the interactions between two algebraic structures. In Poisson geometry, the concept of matched pairs of Lie groupoids and Lie algebroids was studied in \cite{mack,mokri}.
 The matched pairs of Leibniz algebras and 3-Lie algebras were defined in \cite{agore,zhang}. Recently, the representations and cohomology of matched pairs of Lie algebras were investigated in \cite{baishya}, and the matched pairs of (transposed) Poisson 3-Lie algebras were studied in \cite{zhou}.

 On the other hand, a fascinating study related to the expansion of algebraic structures concerns the extendability or inducibility of a pair of automorphisms. Under what conditions can a pair of automorphisms be extended? This question was first posed by Wells \cite{well} in the context of abstract groups and was further explored in \cite{jin, passi}. Since then, several authors have continued to investigate this subject in greater depth.
 The extendability problem for  abelian extensions of Lie algebras and 3-Lie algebas has been addressed in \cite{bar,tan}. As a byproduct of these investigations, the Wells exact sequences for various types of algebras have been obtained \cite{hazra,mishra}, which effectively connect the automorphism groups with the second cohomology groups.

In the first part of this paper, we take the initiative to study a representation of a matched pair of 3-Lie algebras. A representation of a matched pair of 3-Lie algebras $(\mathfrak{g}, \mathfrak{h}, \rho, \psi)$ is given by a quadruple $(V, W, \alpha, \beta)$ in which $V$ and $W$ are both representations of the 3-Lie algebras $\mathfrak{g}$ and $\mathfrak{h}$, and  $\alpha: V \times \mathfrak{g} \rightarrow Hom(\mathfrak{h}, W)$ and $\beta: W \times \mathfrak{h} \rightarrow Hom(\mathfrak{g}, V)$ are bilinear maps  satisfying a set of identities . We present a cohomology theory in the setting of matched pairs of 3-Lie algebras. Subsequently, we explore its applications by examining infinitesimal deformations and abelian extensions of matched pairs of 3-Lie algebras. Lastly, we delve into the inducibility problem of automorphisms within the framework of matched pairs of 3-Lie algebras. To address this issue, we construct Wells exact sequences.

The structure of this paper is outlined as follows. In Section  \ref{sec2}, we introduce the concept of representations of a matched pair of 3-Lie algebras and provide constructions of semi-direct products. Section  \ref{sec4} are devoted to investigating the cohomology of a matched pair of 3-Lie algebras. In Section \ref{sec5}, we delve into infinitesimal deformations and abelian extensions of a matched pair of 3-Lie algebras. Finally, in Sections \ref{sec6} and \ref{sec7}, we develop Wells exact sequences to tackle the inducibility problem of automorphisms.

\section{Representations of a matched pair of 3-Lie algebras}\label{sec2}
In this section, we recall some necessary background on 3-Lie algebras and matched pairs of 3-Lie algebras.
Then, we introduce the concept of representations of a matched pair of 3-Lie algebras and give some examples.
Finally, we define the semidirect product in the context of matched pairs of 3-Lie algebras.

A 3-Lie algebra consists of a vector space $\mathfrak{g}$ together with a linear map $[~, ~, ~] : \wedge^3 \mathfrak{g} \rightarrow \mathfrak{g}$ such that the Jacobi identity.
\begin{align*}
    [x_1, x_2, [y_1, y_2, y_3]]= [[x_1,x_2,y_1], y_2, y_3] + [y_1, [x_1, x_2, y_2], y_3] + [y_1, y_2, [x_1, x_2, y_3]]
\end{align*}
 hold for all $x_1, x_2, y_1, y_2, y_3 \in \mathfrak{g}$ .\\
According to \cite{zhang}, denote by $x=(x_1, x_2)$ and $ad(x)y_i=[x_1, x_2, y_i]$, then the above equality can be rewritten in the form
 \begin{align*}
 ad(x)[y_1, y_2, y_3]= [ad(x)y_1, y_2, y_3] + [y_1, ad(x)y_2, y_3] + [y_1, y_2, ad(x)y_3]
 \end{align*}
 In the following, we recall matched pairs of 3-Lie algebras and the corresponding bicrossed product.

\begin{defn}[\cite{bai}]\label{defn-mpl}
    A {\em matched pair of 3-Lie algebras} is a quadruple $(\mathfrak{g}, \mathfrak{h}, \rho, \psi)$ in which $\mathfrak{g}, \mathfrak{h}$ are both 3-Lie algebras together with linear maps $\rho : \mathfrak{g} \times \mathfrak{g} \times \mathfrak{h} \rightarrow \mathfrak{h}$ and $\psi : \mathfrak{h} \times \mathfrak{h} \times \mathfrak{g} \rightarrow \mathfrak{g}$ satisfying the following conditions:
    \begin{align}
        \psi(a_4, a_5)([x_1, x_2, x_3]) =~& [\psi(a_4, a_5)x_1, x_2, x_3] + [x_1, \psi(a_4, a_5)x_2, x_3] + [x_1, x_2, \psi(a_4, a_5)x_3],\label{11}\\
       -\psi(\rho(x_1, x_2)a_3, a_5)x_4 =~& -\psi(\rho(x_1, x_4)a_5, a_3)x_2 + \psi(\rho(x_2, x_4)a_5, a_3)x_1 - [x_1, x_2, \psi(a_3, a_5)x_4],\label{22}\\
       [\psi(a_2, a_3)x_1, x_4, x_5]   =~& \psi(a_2, a_3)[x_1,x_4, x_5] + \psi(\rho(x_4, x_5)a_2, a_3)x_1 + \psi(a_2, \rho(x_4, x_5)a_3)x_1 , \label{33}\\
        \rho(x_4, x_5)([a_1, a_2, a_3]) =~& [\rho(x_4, x_5)a_1 ,a_2, a_3] + [a_1, \rho(x_4, x_5)a_2, a_3] + [a_1, a_2, \rho(x_4, x_5)a_3],\label{44}\\
       -\rho(\psi(a_1, a_2)x_3, x_5)a_4 =~& -\rho(\psi(a_1, a_4)x_5, x_3)a_2 + \rho(\psi(a_2, a_4)x_5, x_3)a_1 - [a_1, a_2, \rho(x_3, x_5)a_4],\label{55}\\
       [\rho(x_2, x_3)a_1, a_4, a_5]   =~& \rho(x_2, x_3)[a_1,a_4, a_5] + \rho(\psi(a_4, a_5)x_2, x_3)a_1 + \rho(x_2, \psi(a_4, a_5)x_3)a_1 ,\label{66}
    \end{align}
    for $x_1, x_2, x_3, x_4, x_5 \in \mathfrak{g}$ and $a_1, a_2, a_3, a_4, a_5 \in \mathfrak{h}$.
\end{defn}

\begin{thm}[\cite{bai}]\label{thm-bicross}
    Let $(\mathfrak{g}, \mathfrak{h}, \rho, \psi)$ be a matched pair of 3-Lie algebras. Then the direct sum $\mathfrak{g} \oplus \mathfrak{h}$ inherits a 3-Lie algebra structure with the bracket
    \begin{align}\label{bicross}
        [ (x_1, a_1), (x_2, a_2), (x_3, a_3)]_\Join := \big( [x_1, x_2, x_3] + \psi(a_2, a_3)x_1 + \psi(a_3, a_1)x_2 + \psi(a_1, a_2)x_3 ,\nonumber\\
         [a_1, a_2, a_3] + \rho(x_2, x_3)a_1 + \rho(x_3, x_1)a_2 + \rho(x_1, x_2)a_3   \big),
    \end{align}
  for $ (x_1, a_1) , (x_2, a_2) , (x_3, a_3) \in \mathfrak{g} \oplus \mathfrak{h}$.
\end{thm}

The 3-Lie algebra constructed in the above theorem is called the {\em bicrossed product} and it is denoted by $\mathfrak{g} \Join \mathfrak{h}$.

Consider two matched pairs of 3-Lie algebras denoted as $(\mathfrak{g}, \mathfrak{h}, \rho, \psi)$ and $(\mathfrak{g}', \mathfrak{h}', \rho', \psi')$. A morphism from $(\mathfrak{g}, \mathfrak{h}, \rho, \psi)$ to $(\mathfrak{g}', \mathfrak{h}', \rho', \psi')$ is given by a pair of 3-Lie algebra homomorphisms $f: \mathfrak{g} \rightarrow \mathfrak{g}'$ and $g : \mathfrak{h} \rightarrow \mathfrak{h}'$, which satisfy the following conditions
\begin{eqnarray} \label{mpl-mor}
    g (\rho(x_1, x_2)(a)) = {\rho}'(f(x_1, x_2))(g(a)), \label{mpl-mor-1}\\
    f (\psi(a_1, a_2)(a)) = {\psi}'(g(a_1, a_2))(f(x)), \label{mpl-mor-2}
\end{eqnarray}
for all $x, x_1, x_2 \in \mathfrak{g}, a, a_1, a_2 \in \mathfrak{h}$.


\begin{defn}\label{defn-mpl-rep}
    Let $(\mathfrak{g}, \mathfrak{h}, \rho, \psi)$ be a matched pair of 3-Lie algebras. A {\em representation} of $(\mathfrak{g}, \mathfrak{h}, \rho, \psi)$ is given by a quadruple $(V, W, \alpha, \beta)$ in which

    \begin{itemize}
        \item $V$ is a representation for both the 3-Lie algebras $\mathfrak{g}$ and $\mathfrak{h}$ (by the action maps $\rho_V : \mathfrak{g} \times \mathfrak{g} \times V \rightarrow V$ and $\psi_V : \mathfrak{h} \times \mathfrak{h} \times V \rightarrow V$, respectively),

        \item $W$ is a representation for both the 3-Lie algebras $\mathfrak{g}$ and $\mathfrak{h}$ (by the action maps $\rho_W : \mathfrak{g} \times \mathfrak{g} \times W \rightarrow W$ and $\psi_W : \mathfrak{h} \times \mathfrak{h} \times W \rightarrow W$, respectively),

        \item $\alpha: V \times \mathfrak{g} \rightarrow \Hom(\mathfrak{h}, W)$ and $\beta: W \times \mathfrak{h} \rightarrow \Hom(\mathfrak{g}, V)$ are bilinear maps (called the pairing maps) satisfying the following set of identities:
        \begin{align}
            &\psi_V(a_1,a_2)\rho_V(x_2, x_3)v_1\nonumber\\
             =~& \rho_V(\psi(a_1, a_2)x_2, x_3)v_1 + \rho_V(x_2, \psi(a_1, a_2)x_3)v_1 + \rho_V(x_2, x_3)\psi_V(a_1, a_2)v_1, \label{1-iden}\\
            &\psi_V(a_1,a_2)\rho_V(x_1, x_3)v_2 \nonumber\\
            =~& \rho_V(\psi(a_1, a_2)x_1, x_3)v_2 + \rho_V(x_1, \psi(a_1, a_2)x_3)v_2 + \rho_V(x_1, x_3)\psi_V(a_1, a_2)v_2, \label{2-iden}\\
            &\psi_V(a_1,a_2)\rho_V(x_1, x_2)v_3 \nonumber\\
            =~& \rho_V(\psi(a_1, a_2)x_1, x_2)v_3 + \rho_V(x_1, \psi(a_1, a_2)x_2)v_3 + \rho_V(x_1, x_2)\psi_V(a_1, a_2)v_3, \label{3-iden}
         \end{align}
            \begin{align}
            &(\beta(w_1, a_2)-\beta(w_2, a_1))[x_1, x_2, x_3] \nonumber\\
            =~& \rho_V(x_2, x_3)(\beta(w_1, a_2)-\beta(w_2, a_1))x_1 + \rho_V(x_1, x_3)(\beta(w_1, a_2)-\beta(w_2, a_1))x_2 \nonumber\\
            &+ \rho_V(x_1, x_2)(\beta(w_1, a_2)-\beta(w_2, a_1))x_3,\label{4-iden}\\
            &\rho_V(x_2, \psi(a_1, a_2)x_3)v_1 = \psi_V(\rho(x_2, x_3)a_2, a_1)v_1, \label{5-iden}\\
            &\rho_V(x_1, \psi(a_1, a_2)x_3)v_2 = \psi_V(\rho(x_1, x_3)a_2, a_1)v_2, \label{6-iden}\\
            &\rho_V(x_1, x_2)\psi_V(a_1, a_2)v_3 = \psi_V(\rho(x_1, x_2)a_1, a_2)v_3, \label{7-iden}
                     \end{align}
            \begin{eqnarray}
            &&\rho_V(x_1, x_2)(\beta(w_1, a_2)-\beta(w_2, a_1))x_3 \nonumber\\
            &=& \beta(\rho_W(x_2, x_3)w_2 + (\alpha(v_2, x_3)-\alpha(v_3, x_1))a_2, a_1)x_1 - \beta(w_1, \rho(x_2, x_3)a_2)x_1\nonumber\\
            &&- \beta(\rho_W(x_1, x_3)w_2 + (\alpha(v_1, x_3)-\alpha(v_3, x_1))a_2, a_1)x_2 + \beta(w_1, \rho(x_1, x_3)a_2)x_2\nonumber\\
            &&+ \beta(\rho_W(x_1, x_2)w_1 + (\alpha(v_1, x_2)-\alpha(v_2, x_1))a_1, a_2)x_3 - \beta(w_2, \rho(x_1, x_2)a_1)x_3, \label{8-iden}\\
            &&\rho_V(x_2, x_3)\psi_V(a_1, a_2)v_1 \nonumber\\
            &=& \psi_V(a_1, a_2)\rho_V(x_2, x_3)v_1 + \psi_V(\rho(x_2, x_3)a_1, a_2)v_1 + \psi_V(a_1, \rho(x_2, x_3)a_2)v_1, \label{9-iden}\\
            &&\rho_V(x_2, x_3)(\beta(w_1, a_2)-\beta(w_2, a_1))x_1 \nonumber\\
            &=&(\beta(w_1, a_2)-\beta(w_2, a_1))[x_1, x_2, x_3] - \beta(w_2, \rho(x_2, x_3)a_1)x_1 +\beta(w_1, \rho(x_2, x_3)a_2))x_1 \nonumber\\
            &&+ \beta(\rho_W(x_2, x_3)w_1+(\alpha(v_2, x_3)-\alpha(v_3, x_2))a_1, a_2)x_1 \nonumber\\
            &&-\beta(\rho_W(x_2, x_2)w_2+(\alpha(v_2, x_3)-\alpha(v_3, x_2))a_2, a_1)x_1 ,\label{10-iden}\\
           &&\rho_V(\psi(a_1, a_2)x_1, x_3)v_2 = \psi_V(a_1, a_2)\rho_V(x_1, x_3)v_2,\label{11-iden}\\
           &&\rho_V(\psi(a_1, a_2)x_1, x_2)v_3 = \psi_V(a_1, a_2)\rho_V(x_1, x_2)v_3,\label{12-iden}\\
           &&\rho_W(x_1,x_2)\psi_W(a_2, a_3)w_1 \nonumber\\
           &=& \psi_W(\rho(x_1, x_2)a_2, a_3)w_1 + \psi_W(a_2, \rho_(x_1, x_2)a_3)w_1 + \psi_W(a_2, a_3)\rho_W(x_1, x_2)w_1, \label{13-iden}\\
            &&\rho_W(x_1,x_2)\psi_W(a_1, a_3)w_2 \nonumber\\
            &=& \psi_W(\rho(x_1, x_2)a_1, a_3)w_2 + \psi_W(a_1, \rho_(x_1, x_2)a_3)w_2 + \psi_W(a_1, a_3)\rho_W(x_1, x_2)w_2, \label{14-iden}\\
            &&\rho_W(x_1,x_2)\psi_W(a_1, a_2)w_3 \nonumber\\
            &=& \psi_W(\rho(x_1, x_2)a_1, a_2)w_3 + \psi_W(a_1, \rho_(x_1, x_2)a_2)w_3 + \psi_W(a_1, a_2)\rho_W(x_1, x_2)w_3, \label{15-iden}\\
            &&(\alpha(v_1,x_2)-\alpha(v_2, x_1))[a_1, a_2, a_3] \nonumber\\
            &=& \psi_W(a_2, a_3)(\alpha(v_1,x_2)-\alpha(v_2, x_1))a_1 + \psi_W(a_1, a_3)(\alpha(v_1,x_1)+\alpha(v_2, x_2))a_2 \nonumber\\
            &&+\psi_W(a_1, a_2)(\alpha(v_1,x_2)-\alpha(v_2, x_1))a_3,\label{16-iden}\\
            &&\psi_W(a_2, \rho(x_1, x_2)a_3)w_1 = \rho_W(\psi(a_2, a_3)x_2, x_1)w_1, \label{17-iden}
            \end{eqnarray}
            \begin{align}
            &\psi_W(a_1, \rho(x_1, x_2)a_3)w_2 = \rho_W(\psi(a_1, a_3)x_2, x_1)w_2, \label{18-iden}\\
            &\psi_W(a_1, a_2)\rho_W(x_1, x_2)w_3 = \rho_W(\psi(a_1, a_2)x_1, x_2)w_3, \label{19-iden}\\
            &\psi_W(a_1, a_2)(\alpha(v_1,x_2)-\alpha(v_2, x_1))a_3 \nonumber\\
            =~& \alpha(\psi_V(a_2, a_3)v_2 + (\beta(w_2, a_3)-\beta(w_3, a_1))x_2, x_1)a_1 - \alpha(v_1, \rho(a_2, a_3)x_2)a_1\nonumber\\
            ~& -\alpha(\psi_V(a_1, a_3)v_2 + (\beta(w_1, a_3)-\beta(w_3, a_1))x_2, x_1)a_2 + \alpha(v_1, \rho(a_1, a_3)x_2)a_2\nonumber\\
            ~&+ \alpha(\psi_V(a_1, a_2)v_1 + (\beta(w_1, a_2)-\beta(w_2, a_1))x_1, x_2)a_3 - \alpha(v_2, \rho(a_1, a_2)x_1)a_3, \label{20-iden}\\
            &\psi_W(a_2, a_3)\rho_W(x_1, x_2)w_1 \nonumber\\
            =~& \rho_W(x_1, x_2)\psi_V(a_2, a_3)w_1 + \rho_W(\psi(a_2, a_3)x_1, x_2)w_1 + \rho_W(x_1, \psi(a_2, a_3)x_2)w_1, \label{21-iden}\\
           &\psi_W(a_2, a_3)(\alpha(v_1, x_2)-\alpha(v_2, x_1))a_1 \nonumber\\
            =~&(\alpha(v_1, x_2)-\alpha(v_2, x_1))[a_1, a_2, a_3] - \alpha(v_2, \psi(a_2, a_3)x_1)a_1 +\alpha(v_1, \psi(a_2, a_3)x_2))a_1 \nonumber\\
            ~&+ \alpha(\psi_V(a_2, a_3)v_1+(\beta(w_2, a_3)-\beta(w_3, a_2))x_1, x_2)a_1 \nonumber\\
            ~&-\alpha(\psi_V(a_2, a_2)v_2+(\beta(w_2, a_3)-\beta(w_3, a_2))x_2, x_1)a_1,\label{22-iden}\\
            &\psi_W(\rho(x_1, x_2)a_1, a_3)w_2 = \rho_W(x_1, x_2)\psi_W(a_1, a_3)w_2,\label{23-iden}\\
            &\psi_W(\rho(x_1, x_2)a_1, a_2)w_3 = \rho_W(x_1, x_2)\psi_W(a_1, a_2)w_3, \label{24-iden}
         \end{align}
    \end{itemize}
    for $x_1, x_2, x_3 \in \mathfrak{g}$, $a_1, a_2, a_3 \in \mathfrak{h}$, $v_1, v_2, v_3 \in V$ and $w_1, w_2, w_3 \in W$.
\end{defn}

\begin{thm}\label{semid-thm}
Let $(\mathfrak{g}, \mathfrak{h}, \rho, \psi)$ be a matched pair of 3-Lie algebras and $(V, W, \alpha, \beta)$ be a representation of it. Then the quadruple $(\mathfrak{g} \ltimes V, \mathfrak{h} \ltimes W, \rho \ltimes \alpha, \psi \ltimes \beta)$ is a matched pair of 3-Lie algebras, where the maps
    \begin{align*}
        \rho \ltimes \alpha : (\mathfrak{g} \oplus V) \times (\mathfrak{g} \oplus V) \times (\mathfrak{h} \oplus W) \rightarrow \mathfrak{h} \oplus W ~~~~ \text{ and } ~~~~ \psi \ltimes \beta : (\mathfrak{h} \oplus W) \times (\mathfrak{h} \oplus W) \times (\mathfrak{g} \oplus V) \rightarrow \mathfrak{g} \oplus V
    \end{align*}
    are respectively given by
    \begin{align*}
        &\rho \ltimes \alpha((x_1, v_1), (x_2, v_2))(a_1, w_1) := \big( \rho(x_1, x_2)a_1, \rho_W(x_1, x_2)w_1 + (\alpha(v_1, x_2)-\alpha(v_2, x_1))a_1 \big) \\
        &\psi \ltimes \beta((a_1, w_1), (a_2, w_2))(x_1, v_1) := \big( \psi(a_1, a_2)x_1, \psi_V(a_1, a_2)v_1 + (\beta(w_1, a_2)-\beta(w_2, a_1))x_1 \big)
    \end{align*}
\end{thm}

\begin{proof}
    Note that $\mathfrak{g} \ltimes V$ and $\mathfrak{h} \ltimes W$ are both 3-Lie algebras with the 3-Lie brackets (both denoted by the same notation) are respectively given by
    \begin{align*}
        [(x_1, v_1), (x_2, v_2), (x_3, v_3)] := \big( [x_1, x_2, x_3] , \rho_V(x_1, x_2)v_3 + \rho_V(x_3, x_1)v_2 + \rho_V(x_2, x_3)v_1  \big) \\
        [(a_1, w_1), (x_2, w_2), (a_3, w_3)] := \big( [a_1, a_2, a_3] , \psi_W(a_1, a_2)w_3 + \psi_W(a_3, a_1)w_2 + \psi_W(a_2, a_3)w_1  \big),
    \end{align*}
    for $(x_1, v_1), (x_2, v_2), (x_3, v_3) \in \mathfrak{g} \oplus V$ and $(a_1, w_1) , (a_2, w_2), (a_3, w_3) \in \mathfrak{h} \oplus W$. We first claim that the map $\psi \ltimes \beta$ defines a representation of the 3-Lie algebra $\mathfrak{h} \ltimes W$ on the space $\mathfrak{g} \oplus V$. To show this, we observe that
    \begin{eqnarray*}
        &&[\psi \ltimes \beta((a_1, w_1), (a_2, w_2))(x_1, v_1), (x_2, v_2), (x_3, v_3)] \\
        &&+ [(x_1, v_1), \psi \ltimes \beta((a_1, w_1), (a_2, w_2))(x_2, v_2), (x_3, v_3)]  \\
        &&+ [(x_1, v_1), (x_2, v_2), \psi \ltimes \beta((a_1, w_1), (a_2, w_2))(x_3, v_3)] \\
        &=& [(\psi(a_1, a_2)x_1, \psi_V(a_1, a_2)v_1 + (\beta(w_1, a_2)-\beta(w_2, a_1))x_1), (x_2, v_2), (x_3, v_3)]\\
        &&+ [(x_1, v_1), (\psi(a_1, a_2)x_2, \psi_V(a_1, a_2)v_2 + (\beta(w_1, a_2)-\beta(w_2, a_1))x_2), (x_3, v_3)] \\
        &&+ [(x_1, v_1), (x_2, v_2), (\psi(a_1, a_2)x_3, \psi_V(a_1, a_2)v_3 + (\beta(w_1, a_2)-\beta(w_2, a_1))x_3)] \\
        &=& \big( [\psi(a_1, a_2)x_1, x_2, x_3], \rho_V(\psi(a_1, a_2)x_1, x_2)v_3 - \rho_V(\psi(a_1, a_2)x_1, x_3)v_2 \\
        &&+ \rho_V(x_2, x_3)(\psi_V(a_1, a_2)v_1 + (\beta(w_1, a_2)-\beta(w_2, a_1))x_1) \big) + \big( [x_1, \psi(a_1, a_2)x_2, x_3],\\
         &&\rho_V(x_1, \psi(a_1, a_2)x_2)v_3 + \rho_V(x_1, x_3)(\psi_V(a_1, a_2)v_2 + (\beta(w_1, a_2)-\beta(w_2, a_1))x_2) \big) \\
        && +\big( [x_1, x_2, \psi(a_1, a_2)x_3], \rho_V(x_1, x_2)(\psi_V(a_1, a_2)v_3 + (\beta(w_1, a_2)-\beta(w_2, a_1))x_3) \\
        &&+ \rho_V(x_1, \psi(a_1, a_2)x_3)v_2 + \rho_V(x_2, \psi(a_1, a_2)x_3)v_1 \big) \\
        &=& \psi \ltimes \beta((a_1, w_1), (a_2, w_2))[(x_1, v_1), (x_2, v_2), (x_3, v_3)] ,
    \end{eqnarray*}
    for $(x_1, v_1), (x_2, v_2), (x_3, v_3) \in \mathfrak{g} \oplus V$ and $(a_1, w_1), (a_2, w_2) \in \mathfrak{h} \oplus W$. This proves our claim. Similarly, by using the identity (\ref{13-iden}), (\ref{14-iden}), (\ref{15-iden}), (\ref{16-iden}) one can prove that
    \begin{eqnarray*}
        &&[\rho \ltimes \alpha((x_1, v_1), (x_2, v_2))(a_1, w_1), (a_2, w_2), (a_3, w_3)] \\
        &&+ [(a_1, w_1), \rho \ltimes \alpha((x_1, v_1), (x_2, v_2))(a_2, w_2), (a_3, w_3)] \\
        &&+ [ (a_1, w_1), (a_2, w_2), \rho \ltimes \alpha(x_1, v_1), (x_2, v_2)(a_3,w _3)] \\
        &=& \rho \ltimes \alpha((x_1, v_1), (x_2, v_2))[(a_1, w_1), (a_2, w_2), (a_3, w_3)] ,
    \end{eqnarray*}
    for $(a_1, w_1), (a_2,w_2), (a_3, w_3) \in \mathfrak{h} \oplus W$ and $(x_1, v_1), (x_2, v_2) \in \mathfrak{g} \oplus V$. This shows that the map $\rho \ltimes \alpha$ defines a representation of the 3-Lie algebra $\mathfrak{g} \ltimes V$ on the space $\mathfrak{h} \oplus W$. Thus, we are to show the compatibility conditions. For any $(x_1, v_1), (x_2, v_2), (x_3,v_3) \in \mathfrak{g} \oplus V$ and $(a_1, w_1), (a_2, w_2)\in \mathfrak{h} \oplus W$, we have
    \begin{eqnarray*}
       && \psi \ltimes \beta(\rho \ltimes \alpha((x_1, v_1), (x_2, v_2))(a_1, w_1), (a_2, w_2))(x_3, v_3) \\
       &&+  \psi \ltimes \beta(\rho \ltimes \alpha((x_2, v_2), (x_3, v_3))(a_2, w_2), (a_1, w_1))(x_1, v_1) \\
       &&- \psi \ltimes \beta(\rho \ltimes \alpha((x_1, v_1), (x_3, v_3))(a_2, w_2), (a_1, w_1))(x_2, v_2) \\
       &=& \big( \psi(\rho(x_1, x_2)a_1, a_2) x_3, \psi_V(\rho(x_1, x_2)a_1, a_2)v_3 + \beta(\rho_W(x_1, x_2)w_1 + (\alpha(v_1, x_2)-\alpha(v_2, x_1))a_1, a_2)x_3\\
       &&-\beta(w_2, \rho(x_1,x_2)a_1)x_3 \big)+ \big( \psi(\rho(x_2, x_3)a_2, a_1) x_1, \psi_V(\rho(x_2, x_3)a_2, a_1)v_1 \\
       &&+ \beta(\rho_W(x_2, x_3)w_2 + (\alpha(v_2, x_3)-\alpha(v_3, x_2))a_2, a_1)x_1 - \beta(w_1, \rho(x_2,x_3)a_2)x_1 \big) \\
       &&- \big( \psi(\rho(x_1, x_3)a_2, a_1) x_2, \psi_V(\rho(x_1, x_3)a_2, a_1)x_2 + \beta(\rho_W(x_1, x_3)w_2 \\
       &&+ (\alpha(v_1, x_3)-\alpha(v_3,x_1))a_2,a_1)x_2-\beta(w_1,\rho(x_1,x_3)a_2)x_2\big) \\
       &=& [(x_1, v_1), (x_2, v_2)),\psi \ltimes \beta((a_1, w_1), (a_2, w_2))(x_3, v_3)].
    \end{eqnarray*}
    This proves the first compatibility condition. Similarly, by using the identities (\ref{17-iden}), (\ref{18-iden}), (\ref{19-iden}),(\ref{20-iden}) one can prove the second compatibility condition, namely,
    \begin{eqnarray*}
     && \rho \ltimes \alpha(\psi \ltimes \beta((a_1, w_1), (a_2, w_2))(x_1, v_1), (x_2, v_2))(a_3, w_3) \\
     &&+  \rho \ltimes \alpha(\psi \ltimes \beta((a_2, w_2), (a_3, w_3))(x_2, v_2), (x_1, v_1))(a_1, w_1) \\
     &&- \rho \ltimes \alpha(\psi \ltimes \beta((a_1, w_1), (a_3, w_3))(x_2, v_2), (x_1, v_1))(a_2, w_2) \\
     &=&[(a_1, w_1), (a_2, w_2)),\rho \ltimes \alpha((x_1, v_1), (x_2, v_2))(a_3, w_3)].
    \end{eqnarray*}
    for $(a_1, w_1), (a_2, w_2), (a_3, w_3) \in \mathfrak{h} \oplus W$ and $(x_1, v_1), (x_2, v_2) \in \mathfrak{g} \oplus V$. Thus, we are only left to show the last two compatibility conditions.
    For any $(x_1, v_1), (x_2, v_2), (x_3,v_3) \in \mathfrak{g} \oplus V$ and $(a_1, w_1), (a_2, w_2)\in \mathfrak{h} \oplus W$, we have
    \begin{eqnarray*}
    && \psi \ltimes \beta((a_1, w_1), (a_2, w_2))[(x_1, v_1), (x_2, v_2), (x_3,v_3)] \\
    &&+ \psi \ltimes \beta( \rho \ltimes \alpha( (x_2, v_2), (x_3,v_3))(a_1, w_1), (a_2, w_2))(x_1, v_1) \\
    &&+ \psi \ltimes \beta((a_1, w_1), \rho \ltimes \alpha((x_2, v_2), (x_3,v_3))(a_2, w_2))(x_1, v_1) \\
    &=& (\psi(a_1, a_2)[x_1, x_2, x_3], \psi_V(a_1, a_2)(\rho_V(x_1, x_2)v_3 + \rho_V(x_1, x_3)v_2 + \rho_V(x_2, x_3)v_1)\\
    &&+ (\beta(w_1, a_2)-\beta(w_2, a_1))[x_1, x_2, x_3]) + \big(\psi(\rho(x_2,x_3)a_1, a_2)x_1, \psi_V(\rho(x_2, x_3)a_1, a_2)v_1 \\
    &&+ \beta(\rho_W(x_2, x_3)w_1 + (\alpha(v_2, x_3)-\alpha(v_3, x_2))a_1, a_2)x_1-\beta(w_2, \rho(x_2, x_3)a_1)x_1 \big) \\
    &&+ \big( \psi(a_1, \rho(x_2,x_3)a_2)x_1, \psi_V(a_1, \rho(x_2,x_3)a_2)v_1 + \beta(w_1, \rho(x_2, x_3)a_2)x_1\\
    &&-\beta(\rho_W(x_2, x_3)w_2+(\alpha(v_2, x_3)-\alpha(v_3, x_2))a_2, a_1)x_1 \big) \\
    &=& \big([\psi(a_1, a_2)x_1, x_2, x_3], \rho_V(\psi(a_1, a_2)x_1, x_2)v_3 + \rho_V(\psi(a_1, a_2)x_1, x_3)v_2 \\
    &&+ \rho_V(x_2, x_3)\psi(a_1, a_2)x_1 + \rho_V(x_2, x_3)(\beta(w_1, a_2)-\beta(w_2, a_1))x_1 \big)\\
    &=& [\psi \ltimes \beta((a_1, w_1), (a_2, w_2))(x_1, v_1), (x_2, v_2), (x_3,v_3)].
    \end{eqnarray*}
     This proves the other compatibility condition. Similarly, by using the identities (\ref{21-iden}), (\ref{22-iden}), (\ref{23-iden}),(\ref{24-iden}) one can prove the last compatibility condition, namely,
     \begin{eqnarray*}
    && \rho \ltimes \alpha((x_1, v_1), (x_2, v_2))[(a_1, w_1), (a_2, w_2), (a_3,w_3)] \\
    &&+ \rho \ltimes \alpha( \psi \ltimes \beta( (a_2, w_2), (a_3,w_3))(x_1, v_1), (x_2, v_2))(a_1, w_1) \\
    &&+ \rho \ltimes \alpha((x_1, v_1), \psi \ltimes \beta((a_2, w_2), (a_3,w_3))(x_2, v_2))(a_1, w_1) \\
    &=& [\rho \ltimes \alpha((x_1, v_1), (x_2, v_2))(a_1, w_1), (a_2, w_2), (a_3,w_3)].
    \end{eqnarray*}
     for $(a_1, w_1), (a_2, w_2), (a_3, w_3) \in \mathfrak{h} \oplus W$ and $(x_1, v_1), (x_2, v_2) \in \mathfrak{g} \oplus V$.This completes the proof that $(\mathfrak{g} \ltimes V, \mathfrak{h} \ltimes W, \rho \ltimes \alpha, \psi \ltimes \beta)$ is a matched pair of 3-Lie algebras.
\end{proof}

   \section{Cohomology of a matched pair of 3-Lie algebras}\label{sec4}
In this section, we aim to define the cohomology of a matched pair of 3-Lie algebras. We define a low dimensional cohomology of a matched pair of 3-Lie algebras with coefficients in an arbitrary representation. 
\subsection{Cohomology with coefficients in a representation} In this subsection, we introduce the cohomology of a matched pair of 3-Lie algebras with coefficients in a representation. Let $(\mathfrak{g}, \mathfrak{h}, \rho, \psi)$ be a matched pair of 3-Lie algebras and $(V, W, \alpha, \beta)$ be a representation of it. For each $k, l \geq 0$, we consider
\begin{align*}
    C^{k|l} (\mathfrak{g} , \mathfrak{h} ; V , W) := \mathrm{Hom} (\wedge^{k+1} \mathfrak{g} \otimes \wedge^l \mathfrak{h} , V) \oplus \mathrm{Hom} (\wedge^{k} \mathfrak{g} \otimes \wedge^{l+1} \mathfrak{h} , W).
\end{align*}
Next, given any $n \geq 0$, we define the space of $n$-cochains $ C^n (\mathfrak{g} , \mathfrak{h}; V , W)$ by
\begin{align*}
    &C^0 (\mathfrak{g} , \mathfrak{h}; V , W) = V \oplus W ~~~~ \text{ and } \\
    &C^{n \geq 1} (\mathfrak{g} , \mathfrak{h}; V , W) = C^{n-1|0}  (\mathfrak{g} , \mathfrak{h}; V , W) \oplus \cdots \oplus  C^{0|n-1}  (\mathfrak{g} , \mathfrak{h}; V , W).
\end{align*}
Let $C^{1} (\mathfrak{g} , \mathfrak{h}; V , W )=C^{0|0}  (\mathfrak{g} , \mathfrak{h}; V , W )\cong  \mathrm{Hom}(\mathcal{\mathfrak{g}}^{1,0},V) \oplus \mathrm{Hom}(\mathcal{\mathfrak{g}}^{0,1},W)\cong  \mathrm{Hom}(\mathfrak{g},V) \oplus \mathrm{Hom}(\mathfrak{h},W)$ and
 $C^{2} (\mathfrak{g} , \mathfrak{h}; V , W )=C^{1|0}  (\mathfrak{g} , \mathfrak{h}; V , W )\oplus  C^{0|1}  (\mathfrak{g} , \mathfrak{h}; V , W )\cong  \mathrm{Hom}(\mathcal{\mathfrak{g}}^{2,0},V) \oplus \mathrm{Hom}(\mathcal{\mathfrak{g}}^{1,1},W)\oplus \mathrm{Hom}(\mathcal{\mathfrak{g}}^{1,1},V)\oplus \mathrm{Hom}(\mathcal{\mathfrak{g}}^{0,2},W)\cong  \mathrm{Hom}(\mathfrak{g} \otimes \mathfrak{g},V) \oplus \mathrm{Hom}(\mathfrak{g} \otimes \mathfrak{h},W)\oplus\mathrm{Hom}(\mathfrak{h} \otimes \mathfrak{g},W)\oplus\mathrm{Hom}(\mathfrak{g} \otimes \mathfrak{h},V)\oplus \mathrm{Hom}(\mathfrak{h} \otimes \mathfrak{g},V)\oplus\mathrm{Hom}(\mathfrak{h} \otimes \mathfrak{h},W).$
 More precisely, the cochain complex is given by
\begin{equation} \label{eq:complex}
\begin{split}
 & \qquad    V \oplus W \stackrel{D_0}{\longrightarrow}\\
 &  \quad C^{0|0}  (\mathfrak{g} , \mathfrak{h}; V , W )\stackrel{D_1}{\longrightarrow}\\
 & \quad   C^{1|0}  (\mathfrak{g} , \mathfrak{h}; V , W )\oplus
     C^{0|1}  (\mathfrak{g} , \mathfrak{h}; V , W )\stackrel{D_2}{\longrightarrow}\\
  &\quad C^{2|0}  (\mathfrak{g} , \mathfrak{h}; V , W )\oplus
  C^{1|1}  (\mathfrak{g} , \mathfrak{h}; V , W )\oplus
   C^{0|2}  (\mathfrak{g} , \mathfrak{h}; V , W ) \stackrel{D_3}{\longrightarrow}\cdots
\end{split}
\end{equation}

For 1-cochain $ (N_1,N_2)\in  C^{1}  (\mathfrak{g} , \mathfrak{h}; V , W)$, we define the
  map $D_1$ by
\begin{eqnarray*}
D_1 (N_1,N_2)(x_1, x_2, x_3) &=& [x_1, N_1(x_2), x_3] - N_1([x_1, x_2, x_3]) + [N_1(x_1), x_2, x_3] + [x_1, x_2, N_1(x_3)], \notag\\
D_1 (N_1,N_2)(x_1, x_2, a)
&=& \rho(x_1, x_2) N_2(a) - N_2(\rho(x_1, x_2) a) +\alpha(N_1(x_1), x_2)a-\alpha(N_1(x_2), x_1)a,\notag\\
D_1 (N_1,N_2)(a_1, a_2, x)
&=& \psi(a_1, a_2) N_1(x) - N_1(\psi(a_1, a_2) x) + \beta(N_2(a_1), a_2)x-\beta(N_2(a_2), a_1)x,\notag\\
D_1 (N_1,N_2)(a_1, a_2, a_3) &=&  [a_1, N_2(a_2), a_3] - N_2([a_1, a_2, a_3]) + [N_2(a_1), a_2, a_3] + [a_1, a_2, N_2(a_3)] .
\end{eqnarray*}
Thus a 1-cocycle is $(N_1,N_2) \in \mathrm{Hom}(\mathfrak{g},V) \oplus \mathrm{Hom}(\mathfrak{h},W)$ such that
   \begin{eqnarray}
   &&[x_1, N_1(x_2), x_3] - N_1([x_1, x_2, x_3]) + [N_1(x_1), x_2, x_3] + [x_1, x_2, N_1(x_3)]=0,\label{1co-1}\\
   && \rho(x_1, x_2) N_2(a) - N_2(\rho(x_1, x_2) a) + \alpha(N_1(x_1), x_2)a-\alpha(N_1(x_2), x_1)a=0,\label{1co-2}\\
   &&\psi(a_1, a_2) N_1(x) - N_1(\psi(a_1, a_2) x) + \beta(N_2(a_1), a_2)x-\beta(N_2(a_2), a_1)x=0,\label{1co-3}\\
   &&[a_1, N_2(a_2), a_3] - N_2([a_1, a_2, a_3]) + [N_2(a_1), a_2, a_3] + [a_1, a_2, N_2(a_3)]=0,
   \end{eqnarray}
For a 2-cochain
$(\omega, \theta, \nu, \phi)\in  C^{2}  (\mathfrak{g} , \mathfrak{h}; V , W ), $
we define the map $D_2$ by
\begin{eqnarray*}
&&D_2(\omega, \theta, \nu, \phi)(x_1, x_2, x_3, x_4, x_5)\nonumber\\
&=&[ \omega (x_1, x_2, x_3), x_4, x_5] + [ x_3, \omega(x_1, x_2,x_4), x_5] + [ x_3, x_4, \omega(x_1, x_2, x_5)]  \nonumber\\
&&+ \omega ([x_1, x_2, x_3], x_4, x_5) + \omega (x_3, [x_1, x_2, x_4], x_5) + \omega (x_3, x_4, [x_1, x_2, x_5]) \nonumber\\
&&-[x_1, x_2, \omega(x_3, x_4, x_5)] - \omega(x_1, x_2, [x_3, x_4, x_5]) ,   \notag\\
&&D_2(\omega, \theta, \nu, \phi)(x_1, x_2, a_1, a_2, a_3)\nonumber\\
&=& [\nu(x_1, x_2)a_1 ,a_2, a_3] + [a_1, \nu(x_1, x_2)a_2, a_3] + [a_1, a_2, \nu(x_1, x_2)a_3] \nonumber\\
&&+ \theta(\rho(x_1, x_2)a_1, a_2, a_3) + \theta(a_1, \rho(x_1, x_2)a_2, a_3) + \theta(a_1, a_2, \rho(x_1, x_2)a_3)\nonumber\\
&&-\rho(x_1, x_2)\theta(a_1, a_2, a_3) - \nu(x_1, x_2)[a_1, a_2, a_3] ,\notag\\
&&D_2(\omega, \theta, \nu, \phi)(a_1, a_2, x_1, x_2, x_3)\nonumber\\
&=& [\phi(a_1, a_2)x_1 ,x_2, x_3] + [x_1, \phi(a_1, a_2)x_2, x_3] + [x_1, x_2, \phi(a_1, a_2)x_3] \nonumber\\
&&+ \omega(\psi(a_1, a_2)x_1, x_2, x_3) + \omega(x_1, \psi(a_1, a_2)x_2, x_3) + \omega(x_1, x_2, \psi(a_1, a_2)x_3)\nonumber\\
&&-\psi(a_1, a_2)\omega(x_1, x_2, x_3) - \phi(a_1, a_2)[x_1, x_2, x_3] ,\notag\\
&&D_2(\omega, \theta, \nu, \phi)( x_1, x_2, x_3, a_1, a_2)\nonumber\\
&=&\beta(\nu(x_1, x_3)a_2,a_1)x_2 + \phi(\rho(x_1, x_3)a_2, a_1)x_2 - \beta(\nu(x_2, x_3)a_2,a_1)x_1\nonumber\\
&&- \phi(\rho(x_2, x_3)a_2, a_1)x_1 + [x_1, x_2, \phi(a_1, a_2)] + \omega(x_1, x_2, \psi(a_1, a_2)x_3)\nonumber\\
&&-\beta(\nu(x_1, x_2)a_1,a_2)x_3 - \phi(\rho(x_1, x_2)a_1, a_2)x_3 ,\notag\\
&&D_2(\omega, \theta, \nu, \phi)( a_1, a_2, a_3, x_1, x_2)\nonumber\\
&=&\alpha(\phi(a_1, a_3)x_2,x_1)a_2 + \nu(\psi(a_1, a_3)x_2, x_1)a_2 - \alpha(\phi(a_2, a_3)x_2,x_1)a_1\nonumber\\
&&- \nu(\psi(a_2, a_3)x_2, x_1)a_1 + [a_1, a_2, \nu(x_1, x_2)] + \theta(a_1, a_2, \rho(x_1, x_2)a_3)\nonumber\\
&&-\alpha(\phi(a_1, a_2)x_1,x_2)a_3 - \nu(\psi(a_1, a_2)x_1, x_2)a_3 ,\notag\\
&&D_2(\omega, \theta, \nu, \phi)( a_1, a_2, x_1, x_2, x_3)\nonumber\\
&=&\psi(a_1, a_2)\omega(x_1, x_2, x_3) + \phi(a_1, a_2)[x_1, x_2, x_3] + \phi(\rho(x_2, x_3)a_1, a_2)x_1\nonumber\\
&&+\beta(\nu(x_2, x_3)a_1, a_2)x_1 + \phi(a_1, \rho(x_2, x_3)a_2)x_1 - \beta(a_1, \nu(x_2, x_3)a_2)x_1\nonumber\\
&&-\omega(\psi(a_1, a_2)x_1, x_2, x_3) - [\phi(a_1, a_2)x_1, x_2, x_3] ,\notag\\
&&D_2(\omega, \theta, \nu, \phi)( x_1, x_2, a_1, a_2, a_3)\nonumber\\
&=&\rho(x_1, x_2)\theta(a_1, a_2, a_3) + \nu(x_1, x_2)[a_1, a_2, a_3] + \nu(\psi(a_2, a_3)x_1, x_2)a_1\nonumber\\
&&+\alpha(\phi(a_2, a_3)x_1, x_2)a_1 + \nu(x_1, \psi(a_2, a_3)x_2)a_1 - \alpha(x_1, \phi(a_2, a_3)x_2)a_1\nonumber\\
&&-\theta(\rho(x_1, x_2)a_1, a_2, a_3) - [\nu(x_1, x_2)a_1, a_2, a_3] ,\notag\\
&&D_2(\omega, \theta, \nu, \phi)(a_1, a_2, a_3, a_4, a_5)\nonumber\\
&=& [ \theta (a_1, a_2, a_3), a_4, a_5] + [ a_3, \theta(a_1, a_2, a_4), a_5] + [ a_3, a_4, \theta(a_1, a_2, a_5)]  \nonumber\\
&&+ \theta ([a_1, a_2, a_3], a_4, a_5) + \theta (a_3, [a_1, a_2, a_4], a_5) + \theta (a_3, a_4, [a_1, a_2, a_5]) \nonumber\\
&&- [a_1, a_2, \theta(a_3, a_4, a_5)] - \theta(a_1, a_2, [a_3, a_4, a_5]) ,\notag
\end{eqnarray*}
Thus a 2-coboundary is $(\omega, \theta, \nu, \phi) \in C^{2} (\mathfrak{g} , \mathfrak{h}; V , W )$
such that $(\omega, \theta, \nu, \phi)=D_1(N_1,N_2)$, i.e.
   \begin{eqnarray}
   &&\omega(x_1, x_2, x_3)=[x_1, N_1(x_2), x_3] - N_1([x_1, x_2, x_3]) + [N_1(x_1), x_2, x_3] + [x_1, x_2, N_1(x_3)],\\
   &&\nu(x_1, x_2, a_1, a_2, a_3)=\rho(x_1, x_2) N_2(a) - N_2(\rho(x_1, x_2) a) + \alpha(N_1(x_1), x_2)a - \alpha(x_1, N_1(x_2))a,\\
   &&\phi(x_1, x_2, x_3, a_1, a_2)=\psi(a_1, a_2) N_1(x) - N_1(\psi(a_1, a_2) x) + \beta(N_2(a_1), a_2)x - \beta(a_1, N_2(a_2))x,\\
   &&\theta(a_1, a_2, a_3)=[a_1, N_2(a_2), a_3] - N_2([a_1, a_2, a_3]) + [N_2(a_1), a_2, a_3] + [a_1, a_2, N_2(a_3)],
   \end{eqnarray}
and a 2-cocycle is $(\mu_1,\nu_1,\rho_1,\psi_1) \in C^{2} (\mathfrak{g} , \mathfrak{h}; V , W )$ such that
\begin{eqnarray}
&[ \omega (x_1, x_2, x_3), x_4, x_5] + [ x_3, \omega(x_1, x_2,x_4), x_5] + [ x_3, x_4, \omega(x_1, x_2, x_5)]  \nonumber\\
&+ \omega ([x_1, x_2, x_3], x_4, x_5) + \omega (x_3, [x_1, x_2, x_4], x_5) + \omega (x_3, x_4, [x_1, x_2, x_5]) \nonumber\\
&-[x_1, x_2, \omega(x_3, x_4, x_5)] - \omega(x_1, x_2, [x_3, x_4, x_5])=0, \label{2co-1}\\
& [\nu(x_1, x_2)a_1 ,a_2, a_3] + [a_1, \nu(x_1, x_2)a_2, a_3] + [a_1, a_2, \nu(x_1, x_2)a_3] \nonumber\\
&+ \theta(\rho(x_1, x_2)a_1, a_2, a_3) + \theta(a_1, \rho(x_1, x_2)a_2, a_3) + \theta(a_1, a_2, \rho(x_1, x_2)a_3)\nonumber\\
&-\rho(x_1, x_2)\theta(a_1, a_2, a_3) - \nu(x_1, x_2)[a_1, a_2, a_3]=0,\label{2co-2}\\
& [\phi(a_1, a_2)x_1 ,x_2, x_3] + [x_1, \phi(a_1, a_2)x_2, x_3] + [x_1, x_2, \phi(a_1, a_2)x_3] \nonumber\\
&+\omega(\psi(a_1, a_2)x_1, x_2, x_3) + \omega(x_1, \psi(a_1, a_2)x_2, x_3) + \omega(x_1, x_2, \psi(a_1, a_2)x_3)\nonumber\\
&-\psi(a_1, a_2)\omega(x_1, x_2, x_3) - \phi(a_1, a_2)[x_1, x_2, x_3]=0,\label{2co-3}\\
&\beta(\nu(x_1, x_3)a_2,a_1)x_2 + \phi(\rho(x_1, x_3)a_2, a_1)x_2 - \beta(\nu(x_2, x_3)a_2,a_1)x_1\nonumber\\
&- \phi(\rho(x_2, x_3)a_2, a_1)x_1 + [x_1, x_2, \phi(a_1, a_2)] + \omega(x_1, x_2, \psi(a_1, a_2)x_3)\nonumber\\
&-\beta(\nu(x_1, x_2)a_1,a_2)x_3 - \phi(\rho(x_1, x_2)a_1, a_2)x_3 =0, \label{2co-4}\\
&\alpha(\phi(a_1, a_3)x_2,x_1)a_2 + \nu(\psi(a_1, a_3)x_2, x_1)a_2 - \alpha(\phi(a_2, a_3)x_2,x_1)a_1\nonumber\\
&- \nu(\psi(a_2, a_3)x_2, x_1)a_1 + [a_1, a_2, \nu(x_1, x_2)] + \theta(a_1, a_2, \rho(x_1, x_2)a_3)\nonumber\\
&-\alpha(\phi(a_1, a_2)x_1,x_2)a_3 - \nu(\psi(a_1, a_2)x_1, x_2)a_3=0, \label{2co-5}\\
&\psi(a_1, a_2)\omega(x_1, x_2, x_3) + \phi(a_1, a_2)[x_1, x_2, x_3] + \phi(\rho(x_2, x_3)a_1, a_2)x_1\nonumber\\
&+\beta(\nu(x_2, x_3)a_1, a_2)x_1 + \phi(a_1, \rho(x_2, x_3)a_2)x_1 - \beta(a_1, \nu(x_2, x_3)a_2)x_1\nonumber\\
&-\omega(\psi(a_1, a_2)x_1, x_2, x_3) - [\phi(a_1, a_2)x_1, x_2, x_3]=0, \label{2co-6}\\
&\rho(x_1, x_2)\theta(a_1, a_2, a_3) + \nu(x_1, x_2)[a_1, a_2, a_3] + \nu(\psi(a_2, a_3)x_1, x_2)a_1\nonumber\\
&+\alpha(\phi(a_2, a_3)x_1, x_2)a_1 + \nu(x_1, \psi(a_2, a_3)x_2)a_1 - \alpha(x_1, \phi(a_2, a_3)x_2)a_1\nonumber\\
&-\theta(\rho(x_1, x_2)a_1, a_2, a_3) - [\nu(x_1, x_2)a_1, a_2, a_3]=0, \label{2co-7}\\
& [ \theta (a_1, a_2, a_3), a_4, a_5] + [ a_3, \theta(a_1, a_2, a_4), a_5] + [ a_3, a_4, \theta(a_1, a_2, a_5)]  \nonumber\\
&+ \theta ([a_1, a_2, a_3], a_4, a_5) + \theta (a_3, [a_1, a_2, a_4], a_5) + \theta (a_3, a_4, [a_1, a_2, a_5]) \nonumber\\
&- [a_1, a_2, \theta(a_3, a_4, a_5)] - \theta(a_1, a_2, [a_3, a_4, a_5])=0, \label{2co-8}
\end{eqnarray}

Denote the set of 2-cocycles by $\mathcal{Z }^2_\mathrm{MPL} (\mathfrak{g} , \mathfrak{h}; V , W )$, the set of 2-coboundaries by  $ \mathcal{B}^2_\mathrm{MPL} (\mathfrak{g} , \mathfrak{h}; V , W )$ and the second cohomology group by $\mathcal{H}^2_\mathrm{MPL} (\mathfrak{g} , \mathfrak{h}; V , W)$.
It is easy to see that the set of 2-coboundaries is contained in the set of 2-cocycles.

\begin{defn}
   The second cohomology group of $(\mathfrak{g}, \mathfrak{h}, \rho, \psi)$ with coefficients in $(V, W, \alpha, \beta)$ is defined as the quotient
$$\mathcal{H}^2_\mathrm{MPL} (\mathfrak{g} , \mathfrak{h}; V , W )=\mathcal{Z }^2_\mathrm{MPL} (\mathfrak{g} , \mathfrak{h}; V , W )/\mathcal{B }^2_\mathrm{MPL} (\mathfrak{g} , \mathfrak{h}; V , W )$$
\end{defn}

\section{Infinitesimal deformations and abelian extensions}\label{sec5}

\subsection{Infinitesimal deformations} In this subsection, we study infinitesimal deformations of a matched pair of 3-Lie algebras $(\mathfrak{g}, \mathfrak{h}, \rho, \psi)$. Our main result shows that the set of all equivalence classes of infinitesimal deformations is classified by the second cohomology group $H^2_\mathrm{MPL} (\mathfrak{g}, \mathfrak{h}, \rho, \psi)$.

\begin{defn}
    Let $(\mathfrak{g}, \mathfrak{h}, \rho, \psi)$ be a matched pair of 3-Lie algebras. An {\em infinitesimal deformation} of $(\mathfrak{g}, \mathfrak{h}, \rho, \psi)$ is given by a quadruple $(\omega, \theta, \nu, \phi)$ of bilinear maps
    \begin{align*}
        \omega  : \mathfrak{g} \times \mathfrak{g} \times \mathfrak{g} \rightarrow \mathfrak{g}, \quad \theta : \mathfrak{h} \times \mathfrak{h} \times \mathfrak{h} \rightarrow \mathfrak{h}, \quad \nu : \mathfrak{g} \times \mathfrak{g} \times \mathfrak{h} \rightarrow \mathfrak{h} ~~~~ \text{ and } ~~~~ \phi : \mathfrak{h} \times \mathfrak{h}\times \mathfrak{g} \rightarrow \mathfrak{g}
    \end{align*}
    in which $\mu_1, \nu_1$ are skew-symmetric that makes the quadruple $\big( (\mathfrak{g}[t]/(t^2), \mu_t), ( \mathfrak{h}[t]/(t^2), \nu_t), \rho_t, \psi_t \big)$ into a matched pair of 3-Lie algebras over the ring ${\bf k}[t]/(t^2)$. Here the ${\bf k}[t]/(t^2)$-bilinear 3-Lie brackets (on $\mathfrak{g}[t]/(t^2)$ and $\mathfrak{h}[t]/(t^2)$, respectively) and the ${\bf k}[t]/(t^2)$-bilinear maps $\rho_t, \psi_t$ are given by
    \begin{align*}
        \mu_t (x_1, x_2, x_3) :=~& [x_1, x_2, x_3] + t \omega (x_1, x_2, x_3), \\
        \nu_t (a_1, a_2, a_3) :=~& [a_1, a_2, a_3] + t \theta (a_1, a_2, a_3),\\
        \rho_t(x_1, x_2) a_1 :=~& \rho(x_1, x_2) a_1 + t \nu(x_1, x_2) a_1,\\
        \psi_t(a_1, a_2) x_1 :=~& \psi(a_1, a_2) x_1 + t \phi (a_1, a_2) x_1,
    \end{align*}
    for $x_1, x_2, x_3 \in \mathfrak{g}$ and $a_1, a_2, a_3 \in \mathfrak{h}.$
\end{defn}

Let $(\omega, \theta, \nu, \phi)$ be an infinitesimal deformation of the matched pair of 3-Lie algebras $(\mathfrak{g}, \mathfrak{h}, \rho, \psi)$. Since $(\mathfrak{g}[t]/ (t^2), \mu_t)$ is a 3-Lie algebra over the ring ${\bf k}[t]/(t^2)$, it turns out that the skew-symmetric bilinear map $\omega: \mathfrak{g} \times \mathfrak{g} \times \mathfrak{g} \rightarrow \mathfrak{g}$ satisfies
\begin{align}\label{inf-1-eqn}
   &[ \omega (x_1, x_2, x_3), x_4, x_5] + [ x_3, \omega(x_1, x_2,x_4), x_5] + [ x_3, x_4, \omega(x_1, x_2, x_5)]  \nonumber\\
   &+ \omega ([x_1, x_2, x_3], x_4, x_5) + \omega (x_3, [x_1, x_2, x_4], x_5) + \omega (x_3, x_4, [x_1, x_2, x_5]) \nonumber\\
   &- [x_1, x_2, \omega(x_3, x_4, x_5)] - \omega(x_1, x_2, [x_3, x_4, x_5]) =0,
\end{align}
for all $x_1, x_2, x_3, x_4, x_5 \in \mathfrak{g}$. Similarly, $(\mathfrak{h}[t]/(t^2), \nu_t)$ is a 3-Lie algebra over the ring ${\bf k}[t]/(t^2)$ implies that the skew-symmetric bilinear map $\theta : \mathfrak{h} \times \mathfrak{h} \times \mathfrak{h} \rightarrow \mathfrak{h}$ satisfies
\begin{align}\label{inf-2-eqn}
   &[ \theta (a_1, a_2, a_3), a_4, a_5] + [ a_3, \theta(a_1, a_2, a_4), a_5] + [ a_3, a_4, \theta(a_1, a_2, a_5)]  \nonumber\\
   &+ \theta ([a_1, a_2, a_3], a_4, a_5) + \theta (a_3, [a_1, a_2, a_4], a_5) + \theta (a_3, a_4, [a_1, a_2, a_5]) \nonumber\\
   &- [a_1, a_2, \theta(a_3, a_4, a_5)] - \theta(a_1, a_2, [a_3, a_4, a_5]) =0,
\end{align}
for $a_1, a_2, a_3, a_4, a_5 \in \mathfrak{h}$.
Finally, the maps $\rho_t$ and $\psi_t$ satisfies the compatibility conditions(\ref{11})-(\ref{66}) of a matched pair of 3-Lie algebras. These conditions are respectively equivalent to
\begin{eqnarray}
&&\rho(x_1, x_2)\theta(a_1, a_2, a_3) + \nu(x_1, x_2)[a_1, a_2, a_3]\nonumber\\
&=& [\nu(x_1, x_2)a_1 ,a_2, a_3] + [a_1, \nu(x_1, x_2)a_2, a_3] + [a_1, a_2, \nu(x_1, x_2)a_3] \nonumber\\
&&+ \theta(\rho(x_1, x_2)a_1, a_2, a_3) + \theta(a_1, \rho(x_1, x_2)a_2, a_3) + \theta(a_1, a_2, \rho(x_1, x_2)a_3) \label{inf-3-eqn}
\end{eqnarray}
\begin{eqnarray}
&&\psi(a_1, a_2)\omega(x_1, x_2, x_3) + \phi(a_1, a_2)[x_1, x_2, x_3]\nonumber\\
&=& [\phi(a_1, a_2)x_1 ,x_2, x_3] + [x_1, \phi(a_1, a_2)x_2, x_3] + [x_1, x_2, \phi(a_1, a_2)x_3] \nonumber\\
&&+ \omega(\psi(a_1, a_2)x_1, x_2, x_3) + \omega(x_1, \psi(a_1, a_2)x_2, x_3) + \omega(x_1, x_2, \psi(a_1, a_2)x_3)\label{inf-4-eqn}
\end{eqnarray}
\begin{eqnarray}
&&\psi(\nu(x_1, x_2)a_1,a_2)x_3 + \phi(\rho(x_1, x_2)a_1, a_2)x_3\nonumber\\
&=&\psi(\nu(x_1, x_3)a_2,a_1)x_2 +\phi(\rho(x_1, x_3)a_2, a_1)x_2 - \psi(\nu(x_2, x_3)a_2,a_1)x_1\nonumber\\
&&- \phi(\rho(x_2, x_3)a_2, a_1)x_1 + [x_1, x_2, \phi(a_1, a_2)x_3] + \omega(x_1, x_2, \psi(a_1, a_2)x_3)\label{inf-5-eqn}
\end{eqnarray}
\begin{eqnarray}
&&\rho(\phi(a_1, a_2)x_1,x_2)a_3 + \nu(\psi(a_1, a_2)x_1, x_2)a_3\nonumber\\
&=&\rho(\phi(a_1, a_3)x_2,x_1)a_2 + \nu(\psi(a_1, a_3)x_2, x_1)a_2 - \rho(\phi(a_2, a_3)x_2,x_1)a_1\nonumber\\
&&- \nu(\psi(a_2, a_3)x_2, x_1)a_1 + [a_1, a_2, \nu(x_1, x_2)a_3] + \theta(a_1, a_2, \rho(x_1, x_2)a_3)\label{inf-6-eqn}
\end{eqnarray}
\begin{eqnarray}
&&\omega(\psi(a_1, a_2)x_1, x_2, x_3) + [\phi(a_1, a_2)x_1, x_2, x_3]\nonumber\\
&=&\psi(a_1, a_2)\omega(x_1, x_2, x_3) + \phi(a_1, a_2)[x_1, x_2, x_3] + \phi(\rho(x_2, x_3)a_1, a_2)x_1\nonumber\\
&&+\psi(\nu(x_2, x_3)a_1, a_2)x_1 + \phi(a_1, \rho(x_2, x_3)a_2)x_1 + \psi(a_1, \nu(x_2, x_3)a_2)x_1\label{inf-7-eqn}
\end{eqnarray}
\begin{eqnarray}
&&\theta(\rho(x_1, x_2)a_1, a_2, a_3) + [\nu(x_1, x_2)a_1, a_2, a_3]\nonumber\\
&=&\rho(x_1, x_2)\theta(a_1, a_2, a_3) + \nu(x_1, x_2)[a_1, a_2, a_3] + \nu(\psi(a_2, a_3)x_1, x_2)a_1\nonumber\\
&&+\rho(\phi(a_2, a_3)x_1, x_2)a_1 + \nu(x_1, \psi(a_2, a_3)x_2)a_1 + \rho(x_1, \phi(a_2, a_3)x_2)a_1\label{inf-8-eqn}
\end{eqnarray}
for all $x_1, x_2, x_3 \in \mathfrak{g}$ and $a_1, a_2, a_3     \in \mathfrak{h}$.

Let $(\omega, \theta, \nu, \phi)$ and $(\omega', \theta', \nu', \phi')$ be two infinitesimal deformations of a matched pair of 3-Lie algebras $(\mathfrak{g}, \mathfrak{h}, \rho, \psi)$. They are said to be {\em equivalent} (we write $(\omega, \theta, \nu, \phi) \sim (\omega', \theta', \nu', \phi')$) if there exist linear maps $f: \mathfrak{g} \rightarrow \mathfrak{g}$ and $g: \mathfrak{h} \rightarrow \mathfrak{h}$ such that the pair $(\mathrm{id}_\mathfrak{g} + t f, \mathrm{id}_\mathfrak{h} + t g )$ defines a morphism of matched pairs of 3-Lie algebras from
\begin{align*}
   ( (\mathfrak{g}[t]/(t^2) , \mu_t), (\mathfrak{h}[t]/(t^2), \nu_t ), \rho_t, \psi_t) \quad \text{ to } \quad ( (\mathfrak{g}[t]/(t^2), \mu_t') , (\mathfrak{h}[t]/(t^2), \nu_t'), \rho'_t, \psi'_t).
\end{align*}

Suppose $(\omega, \theta, \nu, \phi)$ and $(\omega', \theta', \nu', \phi')$ are two equivalent infinitesimal deformations. Since the map $\mathrm{id}_\mathfrak{g} + t  f : ( \mathfrak{g}[t]/ (t^2), \mu_t) \rightarrow ( \mathfrak{g}[t]/ (t^2), \mu'_t)$ is a morphism of 3-Lie algebras over the ring ${\bf k}[t]/(t^2)$, it follows that
\begin{align}\label{inf-7-eqn}
    \omega (x_1, x_2, x_3) - \omega' (x_1, x_2, x_3) = [x_1, f(x_2), x_3] - f ([x_1, x_2, x_3]) + [f(x_1), x_2, x_3] + [x_1, x_2, f(x_3)],
\end{align}
  for all $ x_1, x_2, x_3 \in \mathfrak{g} $.\\
Similarly, the map $\mathrm{id}_\mathfrak{h} + t  g : ( \mathfrak{h}[t]/ (t^2), \nu_t) \rightarrow ( \mathfrak{h}[t]/ (t^2), \nu'_t)$ is a morphism of 3-Lie algebras over the ring ${\bf k}[t]/(t^2)$ implies that
\begin{align}\label{inf-8-eqn}
  \theta (a_1, a_2, a_3) - \theta' (a_1, a_2, a_3) = [a_1, g(a_2), a_3] - g ([a_1, a_2, a_3]) + [g(a_1), a_2, a_3] + [a_1, a_2, g(a_3)],
\end{align}
 for all $ a_1, a_2, a_3 \in \mathfrak{h}$.\\
Finally, we also have the identities $ (\mathrm{id}_\mathfrak{h} + t  g) ( \rho_t(x_1, x_2) a) = \rho_t'{(\mathrm{id}_\mathfrak{g} + t  f)(x_1, x_2)} (\mathrm{id}_\mathfrak{h} + t  g)(a)$ and $ (\mathrm{id}_\mathfrak{g} + t  f) ( \psi_t(a_1, a_2) x) = \psi_t'{(\mathrm{id}_\mathfrak{h} + t g)(a_1, a_2)} (\mathrm{id}_\mathfrak{g} + t  f)(x)$, for all $x \in \mathfrak{g}$ and $h \in \mathfrak{h}$. These conditions are respectively equivalent to
\begin{align}
    \nu(x_1, x_2) a - \nu'(x_1, x_2) a =~& \rho(x_1, x_2) g (a) - g (\rho(x_1, x_2) a) + \rho(f(x_1), f(x_2))a, \label{inf-9-eqn}\\
    \phi_1(a_1, a_2) x - \phi_1'(a_1, a_2) x =~& \psi(a_1, a_2) f (x) - f (\psi(a_1, a_2) x) + \psi(g(a_1), g(a_2))x,, \label{inf-10-eqn}
\end{align}
for all $x, x_1, x_2 \in \mathfrak{g}$ and $a, a_1, a_2 \in \mathfrak{h}$.
The conditions (\ref{inf-1-eqn})-(\ref{inf-8-eqn}) can be equivalently written as
\begin{align*}
    (\omega, \theta, \nu, \phi) - (\omega', \theta', \nu', \phi') = D_1 ((f,g)),
\end{align*}

This implies that the $2$-cocycles $ (\omega, \theta, \nu, \phi)$ and $(\omega', \theta', \nu', \phi')$ are cohomologous. Hence, we obtain the following result.

\begin{thm}
    Let $(\mathfrak{g}, \mathfrak{h}, \rho, \psi)$ be a matched pair of 3-Lie algebras. Then there is a bijection
    \begin{align*}
        \{ \text{the set of all infinitesimal deformations of } (\mathfrak{g}, \mathfrak{h}, \rho, \psi)\}/\sim ~ ~ \longleftrightarrow ~ H^2_\mathrm{MPL} (\mathfrak{g}, \mathfrak{h}, \rho, \psi).
    \end{align*}
\end{thm}

\medskip

\subsection{Abelian extensions} This subsection is dedicated to exploring abelian extensions of a specified matched pair of 3-Lie algebras through a given representation. We aim to illustrate that the isomorphism classes of these abelian extensions correspond to elements of the second cohomology group of the matched pair of 3-Lie algebras, with coefficients in the representation.

Consider a matched pair of 3-Lie algebras, \((\mathfrak{g},\mathfrak{h},\rho,\psi)\), along with a pair of vector spaces \((V,W)\). It should be highlighted that the quadruple \((V,W,0,0)\) can be construed as a matched pair of 3-Lie algebras, characterized by trivial 3-Lie algebra structures on both \(V\) and \(W\).

\begin{defn}

To describe an abelian extension of a matched pair of 3-Lie algebras $(\mathfrak{g}, \mathfrak{h}, \rho, \psi)$ by a pair of vector spaces $((V,W)$, we consider a specific structure known as a short exact sequence. This sequence is given by:
\begin{align}\label{ab-ext}
\xymatrix{
0 \ar[r] & V \Join W \ar[r]^{i_1 \Join i_2}  &  \widehat{\mathfrak{g}} \Join \widehat{\mathfrak{h}} \ar[r]^{j_1 \Join j_2} & \mathfrak{g} \Join \mathfrak{h} \ar[r] & 0
}
\end{align}
\end{defn}
In this construction, $(\widehat{ \mathfrak{g}}, \widehat{\mathfrak{h}}, \widehat{\rho}, \widehat{\psi})$ represents a new matched pair of 3-Lie algebras. The mappings $i_1 \Join i_2$ and $j_1 \Join j_2$ are morphisms between these matched pairs of 3-Lie algebras, ensuring that the sequence is exact.

We begin by considering a section of the map $j_1 \Join j_2:  \widehat{\mathfrak{g}} \Join \widehat{\mathfrak{h}} \rightarrow \mathfrak{g} \Join \mathfrak{h}$, which is a crucial component in the study of the abelian extension described in (\ref{ab-ext}). This section is defined by a pair of linear maps, $s_1 : \mathfrak{g} \rightarrow \widehat{\mathfrak{g}}$ and $s_2 : \mathfrak{h} \rightarrow \widehat{\mathfrak{h}}$, that satisfy the conditions$j_1 \circ s_1 = \mathrm{id}_\mathfrak{g}$ and $j_2 \circ s_2 = \mathrm{id}_\mathfrak{h}$.

Given such a section $(s_1,s_2)$, we proceed to define a series of maps that play a pivotal role in understanding the structure of the abelian extension. These maps are:
\begin{align*}
    \rho_V : \mathfrak{g} \times \mathfrak{g} \times V \rightarrow V, \quad \rho_W : \mathfrak{g} \times \mathfrak{g} \times W \rightarrow W, \quad \psi_V : \mathfrak{h} \times \mathfrak{h} \times V \rightarrow V \quad \text{ and } \quad \psi_W : \mathfrak{h} \times \mathfrak{h} \times W \rightarrow W
\end{align*}
The definitions of these maps are as follows:
\begin{align*}
    \rho_V(x_1, x_2) v :=~& [s_1(x_1), s_1(x_2), v],   &&\rho_W(x_1, x_2) w := {\widehat{\rho}}(s_1 (x_1),s_1( x_2)) w,\\
    \psi_V(a_1, a_2) v :=~& {\widehat{\psi}}(s_2(a_1), s_2(a_2)) v,  &&\psi_W(a_1, a_2) w := [s_2 (a_1), s_2(a_2), w],
\end{align*}
where \(x_1,x_2 \in\mathfrak{g}\), \(a_1,a_2 \in\mathfrak{h}\), \(v \in V\) and \(w \in W\).

These maps \(\rho_V\) and \(\rho_W\) are shown to define representations of the 3-Lie algebra \(\mathfrak{g}\) on the vector spaces \(V\) and \(W\) respectively. Similarly, \(\psi_V\) and \(\psi_W\) define representations of the 3-Lie algebra \(\mathfrak{h}\) on the vector spaces \(V\) and \(W\).

Furthermore, we introduce additional maps:

$\alpha: V \times \mathfrak{g} \rightarrow Hom(\mathfrak{h}, W)$ and $\beta: W \times \mathfrak{h} \rightarrow Hom(\mathfrak{g}, V)$
defined by:
\begin{align*}
    \alpha(v, x) a := \widehat{\rho}(v, s_1(x))s_2 (a) \quad \text{ and } \quad \beta(w, a) x := \widehat{\psi}(w, s_2(a)) s_1 (x),
\end{align*}
for \(v\in V\), \(a\in\mathfrak{h}\), \(w\in W\) and \(x\in\mathfrak{g}\).

By leveraging the properties of the matched pair of 3-Lie algebras \((\widehat{\mathfrak{g}},\widehat{\mathfrak{h}},\widehat{\rho},\widehat{\psi})\), we can readily prove the following result.

\begin{prop}\label{rep-prop}
With the above notations, the quadruple \((V,W,\alpha,\beta)\) forms a representation of the matched pair of 3-Lie algebras \((\mathfrak{g},\mathfrak{h},\rho,\psi)\).
\end{prop}

Now, consider another section \((s_1',s_2')\) of the map $j_1 \Join j_2$. For any \(x\in\mathfrak{g}\) and \(a\in\mathfrak{h}\), we observe the following relations:
\begin{align*}
    s_1 (x) - s_1' (x) \in \mathrm{ker} (j_1) = \mathrm{im} (i_1) \quad \text{ and } \quad s_2 (a) - s_2'(a)  \in \mathrm{ker} (j_2) = \mathrm{im} (i_2),
\end{align*}

Given these relations, we can analyze the impact on the structures \(\rho_V,\rho_W,\psi_V,\psi_W,\alpha,\beta\) induced on \(V\) and \(W\) by the section \((s_1',s_2')\). Specifically, we find that:
\begin{align*}
    \rho_V(x_1, x_2) v - \rho'_V(x_1, x_2) v =~& [ s_1 (x_1) - s_1' (x_1), s_1 (x_2) - s_1' (x_2), v] = 0,\\
    \rho_W(x_1, x_2) w - \rho'_W(x_1, x_2) w =~& \widehat{\rho}(s_1 (x_1) - s_1' (x_1), s_1 (x_2) - s_1' (x_2)) w = 0,\\
    \psi_V(a_1, a_2) v - \psi'_V(a_1, a_2) v =~& \widehat{\psi}(s_2 (a_1) - s_2' (a_1),s_2 (a_2) - s_2' (a_2) ) v = 0,\\
    \psi_W(a_1, a_2) w - \psi'_W(a_1, a_2) w =~& [ s_2 (a_1) - s_2' (a_1), s_2 (a_2) - s_2' (a_2), w] = 0,\\
    \alpha(v, x) a - \alpha'(v, x) a =~& \widehat{\rho}(v, s_1(x) - s'_1(x))(s_2 (a) - s_2'(a)) = 0,\\
    \beta_(w, a) x - \beta'(w, a) x =~& \widehat{\psi}(w, s_2 (a) - s_2'(a))(s_1 (x) - s_1'(x)) = 0.
\end{align*}
for all \(x_1,x_2\in\mathfrak{g}\), \(v\in V\) and similarly for the other structures.

This analysis reveals that the representation \((V,W,\alpha,\beta)\) is invariant under the choice of section. In other words, the specific section used to define the maps does not affect the resulting representation. This independence is a crucial property, ensuring that the representation is well-defined regardless of the section chosen.

\begin{defn}
In the context of abelian extensions of a matched pair of 3-Lie algebras \((\mathfrak{g},\mathfrak{h},\rho,\psi)\) by a pair of vector spaces \((V,W)\), we define the notion of isomorphism between two such extensions. Specifically, two abelian extensions are considered isomorphic if there exists an isomorphism $f \Join g : \widehat{\mathfrak{g}} \Join \widehat{\mathfrak{h}} \rightarrow \widehat{\mathfrak{g}}' \Join \widehat{\mathfrak{h}}'$ of matched pairs of 3-Lie algebras that makes the following diagram commute:

\begin{align}\label{abel-diag}
        \xymatrix{
0 \ar[r] & V \Join W \ar@{=}[d] \ar[r]^{i_1 \Join i_2}  &  \widehat{\mathfrak{g}} \Join \widehat{\mathfrak{h}} \ar[d]^{f \Join g} \ar[r]^{j_1 \Join j_2} & \mathfrak{g} \Join \mathfrak{h} \ar[r] \ar@{=}[d] & 0 \\
0 \ar[r] & V \Join W \ar[r]_{i_1' \Join i_2'}  &  \widehat{\mathfrak{g}}' \Join \widehat{\mathfrak{h}}' \ar[r]_{j_1' \Join j_2'} & \mathfrak{g} \Join \mathfrak{h} \ar[r] & 0.
}
    \end{align}
\end{defn}

Given a matched pair of 3-Lie algebras \((\mathfrak{g},\mathfrak{h},\rho,\psi)\) and a representation \((V,W,\alpha,\beta)\), we define \(\mathrm{Ext}(\mathfrak{g},\mathfrak{h};V,W)\) as the set of all isomorphism classes of abelian extensions of \((\mathfrak{g},\mathfrak{h},\rho,\psi)\) by the pair of vector spaces \((V,W)\), where the induced representation is the prescribed one. This construction allows us to systematically classify and study the structure of abelian extensions in this context.

\begin{thm}
  Consider a matched pair of 3-Lie algebras \((\mathfrak{g},\mathfrak{h},\rho,\psi)\) along with a representation \((V,W,\alpha,\beta)\). In this context, the isomorphism classes of abelian extensions of \((\mathfrak{g},\mathfrak{h})\) by \((V,W)\) can be classified by the second cohomology group \(H^2_{\mathrm{MPL}}(\mathfrak{g},\mathfrak{h};V,W)\).
\end{thm}

\begin{proof}

Let us start with an abelian extension given by  (\ref{ab-ext}). For any section \((s_1,s_2)\) of the map $j_1 \Join j_2$, we define the following maps:
$ \omega \in \mathrm {Hom} (\wedge^3\mathfrak{g},V) $,$ \theta \in \mathrm {Hom} (\wedge^3 \mathfrak{h},W) $, $ \nu \in \mathrm {Hom} (\mathfrak{g}\otimes \mathfrak{h},W)$, $ \phi \in \mathrm {Hom} (\mathfrak{h}\otimes \mathfrak{g},V)$
by the formulas:
 \begin{align*}
        \begin{cases}
            \omega (x_1, x_2, x_3) := [s_1(x_1), s_1(x_2), s_1(x_3)] - s_1 ([x_1, x_2, x_3]),\\
            \theta (a_1, a_2, a_3) := [s_2(a_1), s_2 (a_2), s_2(a_3)] - s_2 ([a_1, a_2, a_3]),
        \end{cases}\\
        \begin{cases}
        \nu (x_1, x_2)a_1 := \widehat{\rho}(s_1(x_1), s_1(x_2))s_2 (a_1) - s_2 (\rho(x_1, x_2)a_1),\\
            \phi (a_1, a_2)x_1 := \widehat{\psi}(s_2(a_1), s_2(a_2))s_1(x_1) - s_1 (\psi(a_1, a_2)x_1),
        \end{cases}
    \end{align*}
for \(x_1,x_2,x_3\in\mathfrak{g}\) and \(a_1,a_2,a_3\in\mathfrak{h}\).

By a straightforward calculation, it can be verified that the element \((\omega,\theta,\nu,\phi)\in C^2_{\mathrm{MPL}}(\mathfrak{g},\mathfrak{h};V,W)\) is a 2-cocycle. Thus, the abelian  (\ref{ab-ext}) gives rise to a cohomology class in \(H^2_{\mathrm{MPL}}(\mathfrak{g},\mathfrak{h};V,W)\). Moreover, this cohomology class is independent of the choice of the section.
 Let's now turn our attention to two abelian extensions that are considered equivalent, as depicted in diagram (\ref{abel-diag}). Suppose \((s_1,s_2)\) is a section of the map \(j_1\Join j_2\). Then, we observe the following relationships:
\[
j_1\circ(f\circ s_1)=j_1\circ s_1=\mathrm{id}{\mathfrak{g}},\quad j_2\circ(g\circ s_2)=j_2\circ s_2=\mathrm{id}{\mathfrak{h}}.
\]

These relationships imply that \((f\circ s_1,g\circ s_2)\) serves as a section for the map \(j_1\oplus j_2\). Now, let's consider the 2-cocycle \((\omega',\theta',\nu',\phi')\) associated with the second abelian extension and its section \((f\circ s_1,g\circ s_2)\). We can then derive the following:
\begin{align*}
        \omega' (x_1, x_2, x_3) =~& [ (f \circ s_1)(x_1), (f \circ s_1)(x_2), (f \circ s_1)(x_3)] - (f \circ s_1) ([x_1, x_2, x_3]  \\
        =~& f \big(  [s_1(x_1), s_1 (x_2), s_1(x_3)] - s_1 ([x_1, x_2, x_3]) \big) \\
        =~&  [s_1(x_1), s_1 (x_2), s_1(x_3)] - s_1 ([x_1, x_2, x_3]) \quad (\because ~ f|_V = \mathrm{id}_V) \\
        =~& \omega(x_1, x_2, x_3).
    \end{align*}

This derivation shows that the 2-cocycle \(\omega'\) corresponding to the second abelian extension is, in fact, the same as the original 2-cocycle \(\omega\).

In a similar manner, we can show that the other components of the 2-cocycles also match:
 $\nu' (x_1, x_2)a_1 = \nu (x_1, x_2)a_1$,  $\phi' (a_1, a_2)x_1 = \phi (a_1, a_2)x_1$ and $\theta' (a_1, a_2, a_3) = \theta (a_1, a_2, a_3)$.
Thus, we conclude that the 2-cocycles \((\omega,\theta,\nu,\phi)\) and \((\omega',\theta',\nu',\phi')\) are identical. This implies that they correspond to the same element in the cohomology group \(H^2_{\mathrm{MPL}}(\mathfrak{g},\mathfrak{h};V,W)\). Consequently, we have established a well-defined map:
\begin{align*}
    \Theta_1 : \mathrm{Ext}(\mathfrak{g}, \mathfrak{h}; V, W) \rightarrow H^2_\mathrm{MPL} (\mathfrak{g}, \mathfrak{h}; V, W).
\end{align*}

To construct a map in the opposite direction, we begin with a 2-cocycle \((\omega,\theta,\nu,\phi)\) from the cohomology group \(H^2_{\mathrm{MPL}}(\mathfrak{g},\mathfrak{h};V,W)\). Using this 2-cocycle, we define new structures that will help us build an abelian extension.

Specifically, we define new 3-Lie algebras \(\widehat{\mathfrak{g}}=\mathfrak{g}\oplus V\)and\(\widehat{\mathfrak{h}}=\mathfrak{h}\oplus W\). We then introduce the following maps:
\begin{align*}
    [~,~,~]_{\widehat{\mathfrak{g}}} : \widehat{\mathfrak{g}} \times \widehat{\mathfrak{g}}\times \widehat{\mathfrak{g}} \rightarrow \widehat{\mathfrak{g}}, \quad  [~,~,~]_{\widehat{\mathfrak{h}}} : \widehat{\mathfrak{h}} \times \widehat{\mathfrak{h}} \times \widehat{\mathfrak{h}} \rightarrow \widehat{\mathfrak{h}}, \quad \widehat{\rho} : \widehat{\mathfrak{g}} \times \widehat{\mathfrak{h}} \rightarrow \widehat{\mathfrak{h}} \quad \text{ and } \quad \widehat{\psi} : \widehat{\mathfrak{h}} \times \widehat{\mathfrak{g}} \rightarrow \widehat{\mathfrak{g}}
\end{align*}
by
\begin{align*}
     [(x_1, v_1), (x_2, v_2), (x_3, v_3)] =~& \big( [x_1, x_2, x_3] , \rho_V(x_1, x_2)v_3 + \rho_V(x_3, x_1)v_2 + \rho_V(x_2, x_3)v_1 + \omega (x_1, x_2, x_3) \big) \\
     [(a_1, w_1), (x_2, w_2), (a_3, w_3)] =~& \big( [a_1, a_2, a_3] , \psi_W(a_1, a_2)w_3 + \psi_W(a_3, a_1)w_2 + \psi_W(a_2, a_3)w_1 + \theta (a_1, a_2, a_3) \big) \\
     \widehat{\rho}((x_1, v_1), (x_2, v_2))(a_1, w_1) =~& \big( \rho(x_1, x_2)a_1, \rho_W(x_1, x_2)w_1 + (\alpha(v_1, x_2)-\alpha(v_2, x_1))a_1 + \nu (x_1, x_2)a_1 \big) \\
     \widehat{\psi}((a_1, w_1), (a_2, w_2))(x_1, v_1) =~& \big( \psi(a_1, a_2)x_1, \psi_V(a_1, a_2)v_1 + (\beta(w_1, a_2)-\beta(w_2, a_1))x_1 +\phi (a_1, a_2)x_1 \big)
\end{align*}
for $(x_1, v_1), (x_2, v_2), (x_3, v_3) \in \widehat{\mathfrak{g}}$ and $(a_1, w_1), (x_2, w_2), (a_3, w_3) \in \widehat{\mathfrak{h}}$.
These maps are constructed in such a way that \(\widehat{\mathfrak{g}}\) and \(\widehat{\mathfrak{h}}\) become 3-Lie algebras,and \(\widehat{\rho}\) and \(\widehat{\psi}\) define representations of \(\widehat{\mathfrak{g}}\) on \(\widehat{\mathfrak{h}}\) and vice versa. Additionally, these maps satisfy the necessary compatibility conditions, making \((\widehat{\mathfrak{g}},\widehat{\mathfrak{h}},\widehat{\rho},\widehat{\psi})\) a matched pair of 3-Lie algebras.

We then construct an abelian extension:
\begin{align*}
\xymatrix{
0 \ar[r] & V \Join W \ar[r]^{i_1 \Join i_2}  &  \widehat{\mathfrak{g}} \Join \widehat{\mathfrak{h}} \ar[r]^{j_1 \Join j_2} & \mathfrak{g} \Join \mathfrak{h} \ar[r] & 0
}
\end{align*}
where the maps \(i_1,i_2,j_1,\) and \(j_2\) are defined accordingly, where the above maps are given by $i_1 (v) = (0, v)$, $i_2 (w) = (0, w),$ $j_1 (x, v) = x$ and $j_2 (a, w) = a.$

Next, consider another 2-cocycle \((\omega',\theta',\nu',\phi')\) that is cohomologous to \((\omega,\theta,\nu,\phi)\). This means there exists a coboundary \(\delta_{\mathrm{MPL}}(T,\vartheta)\) such that:
\[
(\omega,\theta,\nu,\phi)-(\omega',\theta',\nu',\phi')=\delta_{\mathrm{MPL}}(T,\vartheta),
\]
for some \((T,\vartheta)\in C^1_{\mathrm{MPL}}(\mathfrak{g},\mathfrak{h};V,W)=\mathrm{Hom}(\mathfrak{g},V)\oplus\mathrm{Hom}(\mathfrak{h},W)\).

Using \((T,\vartheta)\), we define linear maps \(f\) and \(g\) that modify the elements of \(\widehat{\mathfrak{g}}\) and \(\widehat{\mathfrak{h}}\) respectively. These maps \(f\) and \(g\) induce an isomorphism between the abelian extensions constructed from \((\omega,\theta,\nu,\phi)\) and \((\omega',\theta',\nu',\phi')\).

Thus, we obtain a well-defined map:
\[
\Theta_2:H^2_{\mathrm{MPL}}(\mathfrak{g},\mathfrak{h};V,W)\rightarrow\mathrm{Ext}(\mathfrak{g},\mathfrak{h};V,W).
\]

Finally, it can be verified that the maps \(\Theta_1\) and \(\Theta_2\) are inverses of each other, completing the proof.
\end{proof}

\section{Inducibility of automorphisms of a matched pair of 3-Lie algebras}\label{sec6}
In this section, we study the inducibility of automorphisms of a matched pair of 3-Lie algebras and characterize them by equivalent conditions.
Assume that ${\mathfrak{g}} \Join {\mathfrak{h}}$ and $V \Join W$ are
two matched pairs of 3-Lie algebras. Let
  \begin{equation*}\mathcal{E}:\xymatrix{
0 \ar[r] & V \Join W \ar[r]^{i_1 \Join i_2}  &  \widehat{\mathfrak{g}} \Join \widehat{\mathfrak{h}} \ar[r]^{j_1 \Join j_2} & {\mathfrak{g}} \Join {\mathfrak{h}} \ar[r] & 0
}\end{equation*}
be an abelian extension of $ {\mathfrak{g}} \Join {\mathfrak{h}}$ by $V \Join W$
   with a section $(s_1,s_2)$ of $(j_1,j_2)$ and
 $(\omega, \theta, \nu,  \phi) $ be the corresponding abelian 2-cocycle.

 Denote
$$\mathrm{Aut}_{V \Join W}
(\widehat{\mathfrak{g}} \Join \widehat{\mathfrak{h}} )=\{\varUpsilon=(\gamma_1,\gamma_2)\in \mathrm{Aut} ( \widehat{\mathfrak{g}} \Join \widehat{\mathfrak{h}})\mid
\gamma_1
(V)\subseteq V,\gamma_2(W)\subseteq W\}.$$
\begin{prop} \label{Idc} The map $\varPhi:\mathrm{Aut}_{V \Join W}
(\widehat{\mathfrak{g}} \Join \widehat{\mathfrak{h}})\longrightarrow \mathrm{Aut} ({\mathfrak{g}} \Join {\mathfrak{h}})\times \mathrm{Aut} (V \Join W)$
given by
\begin{equation}\label{Ik}\varPhi(\varUpsilon)=(\bar{\varUpsilon},\varUpsilon|_{V \Join W})~~ with~~\bar{\varUpsilon}=(\bar{\gamma}_1,\bar{\gamma}_2)=(j_1 \gamma_1 s_1, j_2 \gamma_2 s_2)
,~~\varUpsilon|_{V \Join W}=(\gamma_{1}|_{V},\gamma_{2}|_{W})
\end{equation}
for all $\varUpsilon=(\gamma_1,\gamma_2)\in \mathrm{Aut}_{V \Join W}(\widehat{\mathfrak{g}} \Join \widehat{\mathfrak{h}})$ is well defined. Moreover, $\varPhi$ is a homomorphism of groups.
\end{prop}
\begin{proof}
According to the case of 3-Lie algebras, $\bar{\gamma}_1$
is independent on the choice of sections and $\bar{\gamma}_1$ is an isomorphism of the 3-Lie algebras $\mathfrak{g}$.
By following the same procedure, we can prove that $\bar{\gamma}_2$
does not depend on the choice of sections and $\bar{\gamma}_2$ is bijective. Thus, $\bar{\gamma}$ is independent on the choice of sections. Thanks to $j_2(W)=0$ and $\varUpsilon \in \mathrm{Aut}_{V \Join W}(\widehat{\mathfrak{g}} \Join \widehat{\mathfrak{h}})$, we obtain $j_2 \gamma_2(W)=0$ and $\gamma_2({\widehat{\rho}}(s_1(x_1), s_1(x_2))(s_2(a))={\widehat{\rho}}(\gamma_1 s_1(x_1),\gamma_1 s_1(x_2))(\gamma_2 s_2(a))$.  By direct computations, we have
 \begin{eqnarray*}
 &&\bar{\gamma}_2 (\rho(x_1, x_2)a)\\
&=&j_2 \gamma_2 s_2 (\rho(x_1, x_2)a) \\
&=&j_2 \gamma_2({\widehat{\rho}}(s_1(x_1), s_1(x_2))(s_2(a)) -\nu(x_1, x_2)a)\\
&=&j_2 \gamma_2({\widehat{\rho}}(s_1(x_1), s_1(x_2))(s_2(a)))\\
&=&j_2 {\widehat{\rho}}(\gamma_1 s_1(x_1),\gamma_1 s_1(x_2))(\gamma_2 s_2(a))\\
&=&{\rho}(j_1 \gamma_1 s_1(x_1),j_1 \gamma_1 s_1(x_2) )(j_2 \gamma_2 s_2(a)),
 \end{eqnarray*}
which indicates that  $\bar{\gamma}_2 (\rho(x_1, x_2)a) = {\rho}(\bar{\gamma}_1(x_1),\bar{\gamma}_1(x_2) )(\bar{\gamma}_2(a)).$
 Analogously, we can get  $\bar{\gamma}_1 (\psi(a_1, a_2)x) = {\psi}(\bar{\gamma}_2(a_1),\bar{\gamma}_2(a_2) )(\bar{\gamma}_1(x)).$
These equalities show that, $\bar{\varUpsilon}=(\bar{\gamma}_1,\bar{\gamma}_2)$ is an isomorphism of the matched pair of 3-Lie algebras $\mathfrak{g} \Join \mathfrak{h}$.
It is easy to check that  $\varPhi$ is a homomorphism of groups. This completes the proof.
\end{proof}

A pair $(\alpha,\beta)\in \mathrm{Aut} (\mathfrak{g} \Join \mathfrak{h})\times \mathrm{Aut}(V \Join W)$ is said to
  be inducible if $(\alpha,\beta)$ is an image of $\varPhi$. It is natural to ask: when a pair $(\alpha,\beta)$ is
inducible. We discuss this theme in the following.

\begin{thm} \label{Idu}Assume that $\mathfrak{g} \Join \mathfrak{h}$ and $V \Join W$ are
two matched pairs of 3-Lie algebras. Let
  $\mathcal{E}:\xymatrix{
0 \ar[r] & V \Join W \ar[r]^{i_1 \Join i_2}  &  \widehat{\mathfrak{g}} \Join \widehat{\mathfrak{h}} \ar[r]^{j_1 \Join j_2} & \mathfrak{g} \Join \mathfrak{h} \ar[r] & 0
}$
be an abelian extension of $\mathfrak{g} \Join \mathfrak{h}$ by $V \Join W$
   with a section $(s_1,s_2)$ of $(j_1,j_2)$ and
 $(\omega, \theta, \nu,  \phi) $ be the corresponding abelian 2-cocycle induced by
$(s_1,s_2)$.
 A pair $(\alpha,\beta)\in \mathrm{Aut}(\mathfrak{g} \Join \mathfrak{h})\times \mathrm{Aut}(V \Join W)$ is inducible if and only if there are
linear maps $\zeta:\mathfrak{g}\longrightarrow V$ and $\eta:\mathfrak{h}\longrightarrow W$
satisfying the following conditions:
\begin{align}
&\beta_1(\omega(x_1,x_2, x_3))-\omega(\alpha_1(x_1),\alpha_1(x_2), \alpha_1(x_3))=[\zeta(x_1), \alpha_1(x_2), \zeta(x_3)] \nonumber\\
&-\zeta([x_1, x_2, x_3]) + [\alpha_1(x_1),\zeta(x_2), \zeta(x_3)]+[\zeta(x_1), \zeta(x_2), \alpha_1(x_3)],\label{Iam1}\\
&\beta_2(\theta(a_1, a_2, a_3)) -\theta(\alpha_2(a_1), \alpha_2(a_2), \alpha_2(a_3)) = [\eta(a_1), \alpha_2(a_2), \eta(a_3)] \nonumber\\
&- \eta([a_1, a_2, a_3]) + [\alpha_2(a_1), \eta(a_2), \eta(a_3)] + [\eta(a_1), \eta(a_2), \alpha_2(a_3)],\label{Iam2}\\
&\beta_2(\nu(x_1, x_2) a) - \nu(\alpha_1(x_1),\alpha_1(x_2))(\alpha_2(a)) = \rho(\alpha_1(x_1), \alpha_1(x_2)) \eta(a) \nonumber\\
&-\eta (\rho(x_1, x_2) a) + \alpha(\zeta(x_1), \alpha_1(x_2))(\alpha_2(a))-\alpha(\zeta(x_2), \alpha_1(x_1))(\alpha_2(a)), \label{Iam3}\\
&\beta_1(\phi(a_1, a_2) x) - \phi(\alpha_2(a_1), \alpha_2(a_2)) (\alpha_1(x)) = \psi(\alpha_2(a_1), \alpha_2(a_2)) \zeta (x) \nonumber\\
&- \zeta(\psi(a_1, a_2) x) + \beta(\eta(a_1), \beta(a_2))(\alpha_1(x))-\beta(\eta(a_2), \beta(a_1))(\alpha_1(x)) \label{Iam4}
\end{align}
\end{thm}

\begin{proof} Assume that $(\alpha,\beta)\in \mathrm{Aut}(\mathfrak{g} \Join \mathfrak{h}
)\times \mathrm{Aut}(V \Join W)$ is inducible, that is, there is an
automorphism $\varUpsilon=(\gamma_1,\gamma_2)\in \mathrm{Aut}_{V \Join W}
(\widehat{\mathfrak{g}} \Join \widehat{\mathfrak{h}})$ such
that $\gamma_{1}|_{V}=\beta_1,\gamma_{2}|_{W}=\beta_2$ and $j_2\gamma_2 s_2=\alpha_2, j_1\gamma_1 s_1=\alpha_1$.
Since $(s_1,s_2)$ is a section of $(j_1,j_2)$,
for all
$x\in \mathfrak{g}$ and $a\in \mathfrak{h}$, we have
$$j_1(s_1\alpha_1-\gamma_1 s_1)(x)=\alpha_1(x)-\alpha_1(x)=0,~~ j_2( s_2\alpha_2-\gamma_2  s_2)(a)=\alpha_2(a)-\alpha_2(a)=0,$$
which implies that $(s_1\alpha_1-\gamma_1 s_1)(x)\in \mathrm{ker}(j_1)=V$
and $( s_2\alpha_2-\gamma_2  s_2)(a)\in \mathrm{ker}(j_2)=W$.
So we can define linear maps $\zeta:\mathfrak{g}\longrightarrow
V$ and $\eta:\mathfrak{h}\longrightarrow W$ respectively by
\begin{equation} \zeta(x)=(\gamma_1 s_1-s_1\alpha_1)(x),~~\eta(a)=(\gamma_2 s_2-s_2\alpha_2)(a),~~\forall~x\in \mathfrak{g},~a\in \mathfrak{h}.\end{equation}
Then for any $a \in A, u \in V$,  we have
 \begin{eqnarray*}
&&\beta_2(\nu(x_1, x_2) a) - \nu(\alpha_1(x_1),\alpha_1(x_2))(\alpha_2(a))\\
&=&\gamma_{2}(\nu(x_1, x_2) a) - \nu(\alpha_1(x_1),\alpha_1(x_2))(\alpha_2(a))\\
&=&\gamma_{2}({\widehat{\rho}}(s_1(x_1), s_1(x_2))s_2(a)-s_2(\rho(x_1, x_2)a))\\
&&-{\widehat{\rho}}(s_1(\alpha_1(x_1)), s_1(\alpha_1(x_2)))(s_2(\gamma_2(a)))+s_2(\rho(\alpha_1(x_1),\alpha_1(x_2))(\alpha_2(a)))\\
&=&{\widehat{\rho}}(\gamma_1(s_1(x_1)), \gamma_1(s_1(x_2)))(\gamma_2(s_2(a)))-\gamma_2(s_2(\rho(x_1, x_2)a))\\
&&-{\widehat{\rho}}(\gamma_1(s_1(x_1)),\gamma_1(s_1(x_2)))(s_2(\alpha_2(a)))+{\widehat{\rho}}(\gamma_1(s_1(x_1)),\gamma_1(s_1(x_2)))(s_2(\alpha_2(a)))\\
&&-{\widehat{\rho}}(s_1(\alpha_1(x_1)), s_1(\alpha_1(x_2)))(s_2(\alpha_2(a)))+s_2(\alpha_2(\rho(x_1, x_2)a))\\
&=&{\widehat{\rho}}(\gamma_1(s_1(x_1)), \gamma_1(s_1(x_2)))(\eta(a))- \eta (\rho(x_1, x_2) a)\\
&&+{\widehat{\rho}}(\zeta(x_1),\alpha_1(x_2))(s_2(\alpha_2(a)))-{\widehat{\rho}}(\zeta(x_2),\alpha_1(x_1))(s_2(\alpha_2(a)))\\
&=&\rho(\alpha_1(x_1), \alpha_1(x_2)) \eta(a)- \eta (\rho(x_1, x_2) a) + \alpha(\zeta(x_1), \alpha_1(x_2))(\alpha_2(a))-\alpha(\zeta(x_2), \alpha_1(x_1))(\alpha_2(a)).
\end{eqnarray*}
which implies that Eq.~(\ref{Iam3}) holds.
Analogously, we can show that Eqs.~(\ref{Iam1})-(\ref{Iam2}) and Eqs.(\ref{Iam4}) hold.

Conversely, suppose that $(\alpha,\beta)\in \mathrm{Aut}(\mathfrak{g} \Join \mathfrak{h}
)\times \mathrm{Aut}(V \Join W)$ and there are linear maps $\zeta:\mathfrak{g}\longrightarrow V$ and
$\eta:\mathfrak{h}\longrightarrow W$
satisfying Eqs. (\ref{Iam1})-(\ref{Iam4}). Since $(s_1,s_2)$ is a section of $(j_1,j_2)$,
all $\widehat{x}\in \widehat{\mathfrak{g}},\widehat{a}\in \widehat{\mathfrak{h}}$ can be written as
$\widehat{x}=u+ s_1(x),\widehat{a}=w+s_2(a)$ for some $x\in \mathfrak{g},a\in \mathfrak{h},u\in V,w\in W.$
Define linear maps $\gamma_1:\widehat{\mathfrak{g}}\longrightarrow
\widehat{\mathfrak{g}}$ and $\gamma_2:\widehat{\mathfrak{h}}\longrightarrow
\widehat{\mathfrak{h}}$ respectively by
\begin{equation} \label{Aut1} \gamma_1(\widehat{x})=\gamma_1(u+ s_1(x))=\beta_1(u)+\zeta(x)+ s_1\alpha_1(x),\end{equation}
\begin{equation} \label{Aut2} \gamma_2(\widehat{a})=\gamma_2(w+s_2(a))=\beta_2(w)+\eta(a)+s_2\alpha_2(a).\end{equation}
It is easy to verify that $\gamma_1$ and $\gamma_2$ are bijective. Firstly,
we prove that $\varUpsilon=(\gamma_1,\gamma_2)\in  \mathrm{Aut}( \widehat{\mathfrak{g}}  \Join \widehat{\mathfrak{h}} )$. For any $\widehat{x_1}=u_1+s_1(x_1)$ ,
$\widehat{x_2}= u_2+s_1(x_2)$ and $\widehat{x_3}= u_3+s_1(x_3)$of elememts in $\widehat{\mathfrak{g}}$, we have
\begin{eqnarray*}
&&[\gamma_1(\widehat{x_1}), \gamma_1(\widehat{x_2}), \gamma_1(\widehat{x_3})]\\
&=&[\gamma_1(u_1+s_1(x_1)), \gamma_1(u_2+s_1(x_2)), \gamma_1(u_3+s_1(x_3))]\\
&=&[\beta_1(u_1)+\zeta(x_1)+ s_1\alpha_1(x_1), \beta_1(u_2)+\zeta(x_2)+ s_1\alpha_1(x_2), \beta_1(u_3)+\zeta(x_3)+ s_1\alpha_1(x_3)]\\
&=&[\beta_1(u_1)+\zeta(x_1), \beta_1(u_2)+\zeta(x_2), s_1\alpha_1(x_3)]+[\beta_1(u_1)+\zeta(x_1), s_1\alpha_1(x_2),\beta_1(u_3)+\zeta(x_3)]\\
&&+[\beta_1(u_1)+\zeta(x_1),s_1\alpha_1(x_2), s_1\alpha_1(x_3)]+[s_1\alpha_1(x_1), \beta_1(u_2)+\zeta(x_2), \beta_1(u_3)+\zeta(x_3)]\\
&&+[s_1\alpha_1(x_1), \beta_1(u_2)+\zeta(x_2), s_1\alpha_1(x_3)]+[s_1\alpha_1(x_1), s_1\alpha_1(x_2), \beta_1(u_3)+\zeta(x_3)]\\
&&+\omega(\alpha_1(x_1),\alpha_1(x_2),\alpha_1(x_3))+s_1([\alpha_1(x_1),\alpha_1(x_2),\alpha_1(x_3)])\\
&=&[\beta_1(u_1), \beta_1(u_2), s_1\alpha_1(x_3)]+ [\beta_1(u_1), s_1\alpha_1(x_2), \beta_1(u_3)]+[\beta_1(u_1), s_1\alpha_1(x_2), s_1\alpha_1(x_3)]\\
&&+[s_1\alpha_1(x_1), \beta_1(u_2), \beta_1(u_3)]+[s_1\alpha_1(x_1), \beta_1(u_2), s_1\alpha_1(x_3)]+[s_1\alpha_1(x_1), s_1\alpha_1(x_2), \beta_1(u_3)]\\
&&+\beta(\omega(x_1, x_2, x_3))+\zeta([x_1, x_2, x_3])+s_1\alpha_1([x_1, x_2, x_3])\\
&=&\gamma([u_1, u_2, s_1(x_3)]+[u_1, s_1(x_2), u_3]+[u_1, s_1(x_2), s_1(x_3)]+[s_1(x_1), u_2, u_3]+[s_1(x_1), u_2, x_3]\\
&&+[s_1(x_1), s_1(x_2), u_3]+\omega(x_1, x_2, x_3)+s_1([x_1, x_2, x_3]))\\
&=&\gamma_1[u_1+s_1(x_1), u_2+s_1(x_2), u_1+s_1(x_3)]\\
&=&\gamma_1[\widehat{x_1}, \widehat{x_2}, \widehat{x_3}],
\end{eqnarray*}
which show that $\gamma_1:\widehat{g}\longrightarrow \widehat{g}$ is an algebra morpphism. Similarly, by using the identity, one can prove that $\gamma_2:\widehat{h}\longrightarrow \widehat{h}$ is also an algebra morpphism.
Secondly, we show that the maps $\gamma_1$ and $\gamma_2$ satisfy the compatibilities
\begin{eqnarray*}
&&\gamma_2 ({\widehat{\rho}}(\widehat{x_1}, \widehat{x_2})(\widehat{a})) \\
&=&\gamma_2 ({\widehat{\rho}}(u_1+s_1(x_1), u_2+s_1(x_2))(w+s_2(a))) \\
&=&\gamma_2 ({\widehat{\rho}}(s_1(x_1), s_1(x_2))( w)+{\widehat{\rho}}(u_1, s_1(x_2))(s_2(a))-{\widehat{\rho}}(u_2, s_1(x_1))(s_2(a))+{\widehat{\rho}}(s_1(x_1), s_1(x_2))s_2(a))\\
&=&\gamma_2 (\rho_W(x_1, x_2)w+\beta_2(\alpha(u_1,x_2)a-\alpha(u_2, x_1)a)+\beta_2(\nu(x_1, x_2)a)+\eta(\rho(x_1, x_2)a)+s_2\alpha_2(\rho(x_1, x_2)a))\\
&=&\rho_W(x_1, x_2)\beta_2(w)+\alpha(\beta_1(u_1),\alpha_1(x_2))\alpha_2(a)-\alpha(\beta_1(u_2),\alpha_1(x_1))\alpha_2(a)\\
&&+\nu(\alpha_1(x_1), \alpha_1(x_2))\alpha_2(a)+\rho(\alpha_1(x_1), \alpha_1(x_2))\eta(a)+\alpha(\zeta(x_1), \alpha_1(x_2))(\alpha_2(a))\\
&&-\alpha(\zeta(x_2), \alpha_1(x_1))(\alpha_2(a))+s_2(\rho(\alpha_1(x_1), \alpha_1(x_2))\alpha(a))\\
&=& {{\widehat{\rho}}}(s_1\alpha_1(x_1),s_1\alpha_1(x_2))(\beta_2(w)+\eta(a))+ {{\widehat{\rho}}}(\beta_1(u_1)+\eta(x_1), s_1\alpha_1(x_2))s_2\alpha_2(a)\\
&&-{{\widehat{\rho}}}(\beta_1(u_2)+\eta(x_2), s_1\alpha_1(x_1))s_2\alpha_2(a)+{{\widehat{\rho}}}(s_1\alpha_1(x_1),s_1\alpha_1(x_2))(s_2\alpha_2(a))\\
&=&{{\widehat{\rho}}}(\beta_1(u_1)+\zeta(x_1)+ s_1\alpha_1(x_1), \beta_1(u_2)+\zeta(x_2)+ s_1\alpha_1(x_2))(\beta_2(w)+\eta(a)+s_2\alpha_2(a))\\
&=&{{\widehat{\rho}}}(\gamma_1(\widehat{x_1}), (\widehat{x_2}))(\gamma_2(\widehat{a})),
\end{eqnarray*}
which show that $\gamma_2 ({\widehat{\rho}}(\widehat{x_1}, \widehat{x_2})(\widehat{a})) = {{\widehat{\rho}}}(\gamma_1(\widehat{x_1}), (\widehat{x_2}))(\gamma_2(\widehat{a}))$. Similarly, we can get $\gamma_1 ({\widehat{\psi}}(\widehat{a_1}, \widehat{a_2})(\widehat{x})) = {{\widehat{\psi}}}(\gamma_2(\widehat{a_1}), (\widehat{a_2}))(\gamma_1(\widehat{x}))$.
 Thus, $\varUpsilon=(\gamma_1,\gamma_2)\in  \mathrm{Aut}(\widehat{\mathfrak{g}} \Join \widehat{\mathfrak{h}})$.
Finally, we check that $\gamma_1|_{V}=\beta_1, \gamma_2|_{W}=\beta_2, j_1\gamma_1 s_1=\alpha_1, j_2 \gamma_2 s_2=\alpha_2$.  Indeed, by Eq.~(\ref{Aut1}) and Eq.~(\ref{Aut2}), we have
\begin{equation*}\gamma_1(u)=\gamma_1(u+s_1(0))=\beta_1(u),~\forall~u\in V,\end{equation*}
\begin{equation*}\gamma_2(w)=\gamma_2(w+s_2(0))=\beta_2(w),~\forall~w\in W,\end{equation*}
\begin{equation*}j_1\gamma_1s_1(x)=j_1\gamma_1 (0+s_1(x))=j_1(\zeta(x)+s_1\alpha_1(x))=\alpha_1(x),~\forall~x\in \mathfrak{g},\end{equation*}
\begin{equation*}j_2\gamma_2s_2(a)=j_2\gamma_2(0+s_2(a)=j_2(\eta(a)+s_2\alpha_2(a))=\alpha_2(a),~\forall~a\in \mathfrak{h}.\end{equation*}
This completes the proof.
\end{proof}

Assume that $\mathfrak{g} \Join \mathfrak{h}$ and $V \Join W$ are
two matched pairs of 3-Lie algebras. Let
  \begin{equation*}\mathcal{E}:\xymatrix{
0 \ar[r] & V \Join W \ar[r]^{i_1 \Join i_2}  &  \widehat{\mathfrak{g}} \Join \widehat{\mathfrak{h}} \ar[r]^{j_1 \Join j_2} & g \Join h \ar[r] & 0
}\end{equation*}
be an abelian extension of $\mathfrak{g} \Join \mathfrak{h}$ by $V \Join W$
   with a section $(s_1,s_2)$ of $(j_1,j_2)$ and
 $( \omega, \theta, \nu, \phi) $ be the corresponding abelian 2-cocycle induced by
$(s_1,s_2)$. Define the set $\mathcal{C}$ as follows
\begin{align*}
		\mathcal{C}=&\left\{\begin{aligned}&(\alpha,\beta)\in \mathrm{Aut}( \mathfrak{g} \Join \mathfrak{h}
)\times \mathrm{Aut}(V \Join W ),\\&~~with~~\alpha=(\alpha_1,\alpha_2),~\beta=(\beta_1,\beta_2)
\end{aligned}\left|
\begin{aligned}& Eqs.~(\ref{Iam1})-(\ref{Iam4}) hold
     \end{aligned}\right.\right\}.
 	\end{align*}
The set $\mathcal{C}$ represents of all compatible pairs of automorphisms. Then $\mathcal{C}$ is obviously a subgroup of $\mathrm{Aut}( \mathfrak{g} \Join \mathfrak{h}
)\times \mathrm{Aut}(V \Join W )$. For each pair $(\alpha,\beta)\in \mathrm{Aut}( \mathfrak{g} \Join \mathfrak{h}
)\times \mathrm{Aut}(V \Join W )$, we
define bilinear maps $ \omega_{(\alpha,\beta)} \in \mathrm {Hom} (\mathfrak{g}\otimes \mathfrak{g}\otimes \mathfrak{g},V) $, $ \nu_{(\alpha,\beta)} \in \mathrm {Hom} (\mathfrak{g}\otimes \mathfrak{h},W)$, $ \phi_{(\alpha,\beta)} \in \mathrm {Hom} (\mathfrak{h}\otimes \mathfrak{g},V)$,
  $ \theta_{(\alpha,\beta)} \in \mathrm {Hom} (\mathfrak{h}\otimes \mathfrak{h}\otimes \mathfrak{h},W) $, $\rho_{(\alpha, \beta)} :\mathfrak{g} \times \mathfrak{g} \times W \rightarrow W $,
  $\psi_{(\alpha,\beta)} \mathfrak{h} \times \mathfrak{h} \times V \rightarrow V$, $\alpha_{(\alpha,\beta)}: V \times \mathfrak{g} \rightarrow Hom(h, W)$ and $\beta_{(\alpha,\beta)}: W \times \mathfrak{h} \rightarrow Hom(g, V)$ respectively by
 \begin{align}
\omega_{(\alpha,\beta)}(x_1, x_2, x_3)=\beta_1\omega(\alpha_{1}^{-1}(x_1),\alpha_{1}^{-1}(x_2), \alpha_{1}^{-1}(x_3)),\\
\theta_{(\alpha,\beta)}(a_1, a_2, a_3)=\beta_2\theta(\alpha_{1}^{-1}(a_1),\alpha_{2}^{-1}(a_2), \alpha_{1}^{-1}(a_3)),\\
\nu_{(\alpha,\beta)}(x_1,x_2)(a)=\beta_2\nu(\alpha_{1}^{-1}(x_1),\alpha_{1}^{-1}(x_2))\alpha_{2}^{-1}(a),\\
\phi_{(\alpha,\beta)}(a_1,a_2)(x)=\beta_1\phi(\alpha_{2}^{-1}(a_1),\alpha_{2}^{-1}(a_2))\alpha_{1}^{-1}(x),\\
\rho_{(\alpha,\beta)}(x_1, x_2)(a)=\beta_2\rho(\alpha_{1}^{-1}(x_1),\alpha_{1}^{-1}(x_2))\alpha_{2}^{-1}(a),\\
\psi_{(\alpha,\beta)}(a_1, a_2)(x)=\beta_1\phi(\alpha_{2}^{-1}(a_1),\alpha_{2}^{-1}(a_2))\alpha_{1}^{-1}(x),\\
\alpha_{(\alpha,\beta)}(v,x)(a)=\beta_1\alpha(\beta_2(v),\alpha_{1}^{-1}(x))(\alpha_{2}^{-1}(a)),\\
\beta_{(\alpha,\beta)}(w,a)(x)=\beta_2\beta(\beta_1(w),\alpha_{2}^{-1}(a))(\alpha_{1}^{-1}(x))
\end{align}
for all $x, x_1, x_2, x_3 \in \mathfrak{g}, a, a_1, a_2, a_3 \in \mathfrak{h}, v\in V and w\in W$.

We denote $(\omega_{(\alpha,\beta)},\theta_{(\alpha,\beta)},\nu_{(\alpha,\beta)},\phi_{(\alpha,\beta)} )$ by $(\omega,\theta,\nu,\phi)_{(\alpha,\beta)}$ for simplicity.
In general, $(\omega,\theta,\nu,\phi)_{(\alpha,\beta)}$ may not be a 2-cocycle.
In fact, $(\omega,\theta,\nu,\phi)_{(\alpha,\beta)}$  is a 2-cocycle if $(\alpha,\beta)\in \mathcal{C}$.

\begin{prop} With the notations established above, if $(\alpha,\beta)\in \mathcal{C}$, then $(\omega,\theta,\nu,\phi)_{(\alpha,\beta)}$ forms an abelian 2-cocycle.
\end{prop}
\begin{proof}
Since $(\alpha,\beta)\in \mathcal{C}$, for any $x\in\mathfrak{g} ,a\in \mathfrak{h}$, we have the following equalities
$ \alpha_{2}^{-1}(\rho(x_1, x_2)a)=\rho(\alpha_{1}^{-1}(x_1), \alpha_{1}^{-1}(x_12)\alpha_{2}^{-1}(a)$ .
Given that $(\omega,\theta,\nu,\phi)$ is a 2-cocycle, the identities (\ref{2co-1})-(\ref{2co-4}) are hold.
In identities (\ref{2co-2}), if we replace $x_1, x_2, a_1, a_2, a_3$ by $\alpha_{1}^{-1}(x_1),\alpha_{1}^{-1}(x_2), \alpha_{2}^{-1}(a_1), \alpha_{2}^{-1}(a_2), \alpha_{2}^{-1}(a_3)$ respectively, we obtain
 \begin{eqnarray*}
  &\beta_2( [\nu(\alpha_{1}^{-1}(x_1), \alpha_{1}^{-1}(x_2))\alpha_{2}^{-1}(a_1) ,\alpha_{2}^{-1}(a_2), \alpha_{2}^{-1}(a_3)] + [\alpha_{2}^{-1}(a_1), \nu(\alpha_{1}^{-1}(x_1), \alpha_{1}^{-1}(x_2))\alpha_{2}^{-1}(a_2), \alpha_{2}^{-1}(a_3)] \\
  &+ [\alpha_{2}^{-1}(a_1), \alpha_{2}^{-1}(a_2), \nu(\alpha_{1}^{-1}(x_1), \alpha_{1}^{-1}(x_2))\alpha_{2}^{-1}(a_3)] + \theta(\rho(\alpha_{1}^{-1}(x_1),\alpha_{1}^{-1}(x_2))\alpha_{2}^{-1}(a_1), \alpha_{2}^{-1}(a_2), \alpha_{2}^{-1}(a_3)) \\
  &+ \theta(\alpha_{2}^{-1}(a_1), \rho(\alpha_{1}^{-1}(x_1), \alpha_{1}^{-1}(x_2))\alpha_{2}^{-1}(a_2), \alpha_{2}^{-1}(a_3)) + \theta(\alpha_{2}^{-1}(a_1), \alpha_{2}^{-1}(a_2), \rho(\alpha_{1}^{-1}(x_1), \alpha_{1}^{-1}(x_2))\alpha_{2}^{-1}(a_3))\\
&-\rho(\alpha_{1}^{-1}(x_1), \alpha_{1}^{-1}(x_2))\theta(\alpha_{2}^{-1}(a_1), \alpha_{2}^{-1}(a_2), \alpha_{2}^{-1}(a_3)) - \nu(\alpha_{1}^{-1}(x_1), \alpha_{1}^{-1}(x_2))[\alpha_{2}^{-1}(a_1), \alpha_{2}^{-1}(a_2), \alpha_{2}^{-1}(a_3)])=0,
  \end{eqnarray*}
  which can be written as
   \begin{eqnarray*}
  &[\beta_2\nu(\alpha_{1}^{-1}(x_1), \alpha_{1}^{-1}(x_2))\alpha_{2}^{-1}(a_1) ,a_2, a_3] + [a_1, \beta_2\nu(\alpha_{1}^{-1}(x_1), \alpha_{1}^{-1}(x_2))\alpha_{2}^{-1}(a_2), a_3] \\
  &+ [a_1, a_2, \beta_2\nu(\alpha_{1}^{-1}(x_1), \alpha_{1}^{-1}(x_2))\alpha_{2}^{-1}(a_3)] + \beta_2\theta(\alpha_{2}^{-1}(\rho(x_1,x_2)(a_1)), \alpha_{2}^{-1}(a_2), \alpha_{2}^{-1}(a_3)) \\
  &+ \beta_2\theta(\alpha_{2}^{-1}(a_1), \alpha_{2}^{-1}(\rho(x_1, x_2)(a_2)), \alpha_{2}^{-1}(a_3)) + \beta_2\theta(\alpha_{2}^{-1}(a_1), \alpha_{2}^{-1}(a_2), \alpha_{2}^{-1}(\rho(x_1, x_2)(a_3))\\
&-\beta_2\rho(\alpha_{1}^{-1}(x_1), \alpha_{1}^{-1}(x_2))\theta(\alpha_{2}^{-1}(a_1), \alpha_{2}^{-1}(a_2), \alpha_{2}^{-1}(a_3)) \\
&- \beta_2\nu(\alpha_{1}^{-1}(x_1), \alpha_{1}^{-1}(x_2))[\alpha_{2}^{-1}(a_1), \alpha_{2}^{-1}(a_2), \alpha_{2}^{-1}(a_3)]=0,
  \end{eqnarray*}
This is equivalent to
\begin{eqnarray*}
  &[\nu_{(\alpha,\beta)}(x_1, x_2)a_1 ,a_2, a_3] + [a_1, \nu_{(\alpha,\beta)}(x_1,x_2)a_2, a_3] + [a_1, a_2, \nu_{(\alpha,\beta)}(x_1, x_2)a_3] \\
  &+\theta_{(\alpha,\beta)}(\rho(x_1,x_2)(a_1), a_2, a_3) + \theta_{(\alpha,\beta)}(a_1, \rho(x_1, x_2)(a_2), a_3) + \theta_{(\alpha,\beta)}(a_1, a_2, \rho(x_1, x_2)(a_3))\\
&-\rho_{(\alpha,\beta)}(x_1, x_2)\theta_{(\alpha,\beta)}(a_1, a_2, a_3)- \nu_{(\alpha,\beta)}(x_1, x_2)[a_1, a_2, a_3]=0,
 \end{eqnarray*}
 which implies that Eq.~(\ref{2co-2}) holds for $(\omega,\theta,\nu,\phi)_{(\alpha,\beta)}$. Similarly, we can check that
Eqs.~(\ref{2co-1})  and (\ref{2co-3})-(\ref{2co-4}) hold. This finishes the proof.
 \end{proof}

 \begin{thm} \label{CIdu}Assume that $\mathfrak{g} \Join \mathfrak{h}$ and $V \Join W$ are
two matched pairs of 3-Lie algebras. Let
  $\mathcal{E}:\xymatrix{
0 \ar[r] & V \Join W \ar[r]^{i_1 \Join i_2}  &  \widehat{\mathfrak{g}} \Join \widehat{\mathfrak{h}} \ar[r]^{j_1 \Join j_2} & \mathfrak{g} \Join \mathfrak{h} \ar[r] & 0
}$
be an abelian extension of $\mathfrak{g} \Join \mathfrak{h}$ by $V \Join W$
   with a section $(s_1,s_2)$ of $(j_1,j_2)$ and
 $(\omega,\theta,\nu,\phi) $ be the corresponding abelian 2-cocycle induced by
$(s_1,s_2)$.
 A pair $ (\alpha,\beta)$ is inducible if and only if $(\alpha,\beta)\in \mathcal{C}$ and the two
 abelian 2-cocycles  $(\omega,\theta,\nu,\phi)_{(\alpha,\beta)}$ and  $(\omega,\theta,\nu,\phi)$
 are equivalent.
\end{thm}
\begin{proof}
Let $(\alpha,\beta)$ is inducible. It is obviously that $(\alpha,\beta)\in \mathcal{C}$. Then by Theorem \ref{Idu}, there exists linear maps $\zeta:\mathfrak{g}\longrightarrow V$ and $\eta:\mathfrak{h}\longrightarrow W$
satisfying Eqs.~ (\ref{Iam1})-(\ref{Iam4}).
 We replace $x, x_1, x_2, x_3, a, a_1, a_2, a_3$ by $\alpha_{1}^{-1}(x),\alpha_{1}^{-1}(x_1), \alpha_{1}^{-1}(x_2), \alpha_{1}^{-1}(x_3), \alpha_{2}^{-1}(a),\alpha_{2}^{-1}(a_1), \alpha_{2}^{-1}(a_2), \alpha_{2}^{-1}(a_3)$ respectively, we get
 \begin{align*}
&\omega_{(\alpha,\beta)}(x_1, x_2, x_3)-\omega(x_1,x_2,x_3)=[\zeta\alpha_{1}^{-1}(x_1), x_2, \zeta\alpha_{1}^{-1}(x_3)] \\
&-\zeta\alpha_{1}^{-1}([x_1, x_2, x_3])+ [x_1,\zeta\alpha_{1}^{-1}(x_2), \zeta\alpha_{1}^{-1}(x_3)]+[\zeta\alpha_{1}^{-1}(x_1), \zeta\alpha_{1}^{-1}(x_2), x_3],\\
&\theta_{(\alpha,\beta)}(a_1, a_2, a_3))-\theta(a_1, a_2, a_3) = [\eta\alpha_{2}^{-1}(a_1), a_2, \eta\alpha_{2}^{-1}(a_3)] \\
&- \eta\alpha_{2}^{-1}([a_1, a_2, a_3])+ [a_1, \eta\alpha_{2}^{-1}(a_2), \eta\alpha_{2}^{-1}(a_3)] + [\eta\alpha_{2}^{-1}(a_1), \eta\alpha_{2}^{-1}(a_2), a_3],\\
&\nu_{(\alpha,\beta)}(x_1, x_2) a - \nu(x_1, x_2)(a)= \rho(x_1, x_2) \eta\alpha_{2}^{-1}(a) \\
&-\eta (\rho(\alpha_{1}^{-1}(x_1), \alpha_{1}^{-1}(x_2)) \alpha_{2}^{-1}(a))+ \alpha(\zeta\alpha_{1}^{-1}(x_1),x_2)(a)-\alpha(\zeta\alpha_{1}^{-1}(x_2),x_1)(a), \\
&\phi{(\alpha,\beta)}(a_1, a_2) x - \phi(a_1, a_2) (x) = \psi(a_1, a_2) \zeta\alpha_{1}^{-1} (x)\\
&- \zeta(\psi(\alpha_{2}^{-1}(a_1), \alpha_{2}^{-1}(a_2)) \alpha_{1}^{-1}(x))+ \beta(\eta\alpha_{2}^{-1}(a_1),a_2)(x)-\beta(\eta\alpha_{2}^{-1}(a_2), a_1)(x) .
\end{align*}
This shows that  $(\omega,\theta,\nu,\phi)_{(\alpha,\beta)}$ and  $(\omega,\theta,\nu,\phi)$
 are equivalent and the equivalence is given by the map $(\zeta\alpha_{1}^{-1},\eta\alpha_{2}^{-1}) \in \mathrm{Hom}(\mathfrak{g},V) \oplus \mathrm{Hom}(\mathfrak{h},W).$

 Conversely, suppose that $(\omega,\theta,\nu,\phi)_{(\alpha,\beta)}$ and  $(\omega,\theta,\nu,\phi)$
 are equivalent and the equivalence is given by the map $(f,g) \in \mathrm{Hom}(\mathfrak{g},V) \oplus \mathrm{Hom}(\mathfrak{h},W).$
 Then it can be easily checked that the maps $\zeta:=f \alpha_{1}:\mathfrak{g}\longrightarrow V$  and $\eta:=g \alpha_{2}:\mathfrak{h}\longrightarrow W$ satisfies the Eqs.~ (\ref{Iam1})-(\ref{Iam4}). Thus the pair $(\alpha,\beta)$ is inducible.
\end{proof}
\section{Wells exact sequences}\label{sec7}
In this section, we consider the Wells map associated with abelian extensions of a matched pair of 3-Lie algebras.

Let
  \begin{equation*}\mathcal{E}:\xymatrix{
0 \ar[r] & V \Join W \ar[r]^{i_1 \Join i_2}  &  \widehat{\mathfrak{g}} \Join \widehat{\mathfrak{h}} \ar[r]^{j_1 \Join j_2} & \mathfrak{g} \Join \mathfrak{h} \ar[r] & 0
}\end{equation*}
be an abelian extension of $\mathfrak{g} \Join \mathfrak{h}$ by $V \Join W$
   with a section $(s_1,s_2)$ of $(j_1,j_2)$ and
 $(\omega,\theta,\nu,\phi) $ be the corresponding abelian 2-cocycle.

Define a  map $\mathcal {W}:\mathcal{C}\longrightarrow \mathcal{H}^2_\mathrm{MPL} (\mathfrak{g} \Join \mathfrak{h}, V \Join W)$ by
\begin{equation}\label{W1}
	\mathcal {W}(\alpha,\beta)
=[(\omega,\theta,\nu,\phi)_{(\alpha,\beta)}
-(\omega,\theta,\nu,\phi)].
\end{equation}
The map $\mathcal {W}$ is called the Wells map associated with $\mathcal{E}$.

\begin{prop}
 The Wells map  $\mathcal {W}$  does not depend on the choice of section.
\end{prop}
\begin{proof}

 Let $(s_1,s_2)$ and $(s_1',s_2')$ be two
sections of $(j_1,j_2)$. Let $(\omega,\theta,\nu,\phi)$
be the corresponding  abelian 2-cocycle induced by the section $(s_1, s_2)$ and  $(\omega',\theta',\nu',\phi')$
be the corresponding  abelian 2-cocycle induced by the section $(s_1', s_2')$. Define linear maps $\zeta:
\mathfrak{g}\longrightarrow V,~\eta:\mathfrak{h}\longrightarrow W$ respectively by $\zeta(x)=s_1(x)-s_1'(x),~\eta(a)=s_2(a)-s_2'(a)$. Since
$j_1 \zeta(x)=j_1s_1(x)-j_1s_1'(x)=0,~j_2\eta(a)=j_2s_2(a)-j_2s_2'(a)=0$,
 $\zeta,\eta$ are well defined. Then we have that the induced abelian 2-cocycles $(\omega,\theta,\nu,\phi)$ and $(\omega',\theta',\nu',\phi')$ are equivalent, and the equivalence is given by the map $(\zeta,\eta)$.

On the other hand, the abelian 2-cocycles $(\omega,\theta,\nu,\phi)_{(\alpha,\beta)}$ and $(\omega',\theta',\nu',\phi') $ are equivalent, and the equivalence is given by $(\beta_1 \zeta \alpha_1^{-1},\beta_2 \eta \alpha_2^{-1}) $.

Combining the results of the last two paragraphs, we have that the abelian 2-cocycles $$(\omega,\theta,\nu,\phi)_{(\alpha,\beta)}- (\omega,\theta,\nu,\phi)$$
and $$(\omega',\theta',\nu',\phi')_{(\alpha,\beta)}- (\omega',\theta',\nu',\phi')$$
are equivalent, and equivalence is given by $(\beta_1 \zeta \alpha_1^{-1}-\zeta,\beta_2 \eta \alpha_2^{-1}-\eta)$. Hence they corresponds to the same element in $\mathcal{H}^2_\mathrm{MPL} (\mathfrak{g} \Join \mathfrak{h},V \Join W)$. This completes the proof.
\end{proof}
It follows from Theorem \ref{Idu} and \ref{CIdu} that
 $(\alpha,\beta)\in \mathrm{Aut}(\mathfrak{g} \Join \mathfrak{h}
)\times \mathrm{Aut}(V \Join W)$ is inducible if and only if $(\alpha,\beta)\in \mathcal{C}$ and $\mathcal {W}(\alpha,\beta)=0$.

Let
\begin{align*}
\mathrm{Aut}_{V \Join W}^{\mathfrak{g} \Join \mathfrak{h}}
(\widehat{\mathfrak{g}} \Join \widehat{\mathfrak{h}})=\{\varUpsilon \in \mathrm{Aut}(\widehat{\mathfrak{g}} \Join \widehat{\mathfrak{h}})| \varPhi(\varUpsilon)=(\mathrm{Id} _{\mathfrak{g} \Join \mathfrak{h}},\mathrm{Id}_{V \Join W})\},
 \end{align*}
where $\mathrm{Id}_{\mathfrak{g} \Join \mathfrak{h}}=(\mathrm{id}_\mathfrak{g},\mathrm{id}_\mathfrak{h})$, $\mathrm{Id}_{V \Join W}=(\mathrm{id}_V,\mathrm{id}_W)$. Recall that
\begin{align*}
		 \mathcal{Z}^1_\mathrm{MPL} (\mathfrak{g} \Join \mathfrak{h},V \Join W)=&\left\{(\zeta,\eta),~\zeta:\mathfrak{g}\rightarrow V,~\eta:\mathfrak{h}\rightarrow W \left|\begin{aligned}& D^1(\zeta,\eta)=0
     \end{aligned}\right.\right\}.
     \end{align*}
It is easy to check that $ \mathcal{Z}^1_\mathrm{MPL}(\mathfrak{g} \Join \mathfrak{h},V \Join W)$ is an abelian group, which is called an abelian 1-cocycle on $\mathfrak{g} \Join \mathfrak{h}$ with values in $V \Join W$.
\begin{prop}\label{Der}
Let
  $\mathcal{E}:\xymatrix{
0 \ar[r] & V \Join W \ar[r]^{i_1 \Join i_2}  &  \widehat{\mathfrak{g}} \Join \widehat{\mathfrak{h}} \ar[r]^{j_1 \Join j_2} & \mathfrak{g} \Join \mathfrak{h} \ar[r] & 0
}$
be an abelian extension of $\mathfrak{g} \Join \mathfrak{h}$ by $V \Join W$. Then $ \mathcal{Z}^1_\mathrm{MPL} (\mathfrak{g} \Join \mathfrak{h},V \Join W) \cong \mathrm{Aut}_{V \Join W}^{\mathfrak{g} \Join \mathfrak{h}}(\widehat{\mathfrak{g}} \Join \widehat{\mathfrak{h}})$ as groups.
\end{prop}
\begin{proof}
Define $\varPsi:\mathrm{Aut}_{V \Join W}^{\mathfrak{g} \Join \mathfrak{h}}(\widehat{\mathfrak{g}} \Join \widehat{\mathfrak{h}})\longrightarrow \mathcal{Z}^1_\mathrm{MPL}(\mathfrak{g} \Join \mathfrak{h},V \Join W)$  by
  $\varPsi(\varUpsilon)=(\varPsi(\gamma_1),\varPsi(\gamma_2))$, where $\varPsi(\gamma_1) \triangleq \zeta_{\gamma_1}:\mathfrak{g} \to V$ and $\varPsi(\gamma_2) \triangleq \eta_{\gamma_2}:\mathfrak{h} \to W$ are given by
  $$\varPsi(\gamma_1)(x)=\zeta_{\gamma_1}(x)=\gamma_1 s_1(x)-s_1(x),~~
 \varPsi(\gamma_2)(a)=\eta_{\gamma_2}(a)=\gamma_2 s_2(a)-s_2(a), $$
  for all $\varUpsilon=(\gamma_1,\gamma_2)\in \mathrm{Aut}_{V \Join W}^{\mathfrak{g} \Join \mathfrak{h}}(\widehat{\mathfrak{g}} \Join \widehat{\mathfrak{h}}),~x \in \mathfrak{g},~a \in \mathfrak{h}.$
Firstly, we prove that $\varPsi$ is well-defined. By direct computation, we have
\begin{eqnarray*}
 &&\rho(x_1, x_2)\eta_{\gamma_2}(x)-\eta_{\gamma_2}(\rho(x_1, x_2)a)+\alpha(\zeta_{\gamma_1}(x_1), x_2)a-\alpha(\zeta_{\gamma_1}(x_2), x_1)a\\
 &=&{\widehat{\rho}}(s_1(x_1), s_1(x_2))\eta_{\gamma_2}(a)-\eta_{\gamma_2}(\rho(x_1, x_2)a)
 +{\widehat{\rho}}(\zeta_{\gamma_1}(x_1), s_1(x_2))(s_2(a))-{\widehat{\rho}}(\zeta_{\gamma_1}(x_2), s_1(x_1))(s_2(a))\\
 &=&{\widehat{\rho}}(s_1(x_1), s_1(x_2))\eta_{\gamma_2}(a)  +{\widehat{\rho}}(\zeta_{\gamma_1}(x_1), s_1(x_2))(s_2(a))-{\widehat{\rho}}(\zeta_{\gamma_1}(x_2), s_1(x_1))(s_2(a))\\
 &&-\gamma_2 {\widehat{\rho}}(s_1(x_1),s_1(x_2))(s_2(a))+{\widehat{\rho}}(s_1(x_1), s_1(x_2))(s_2(a))+\gamma_2 {\widehat{\rho}}(s_1(x_1),s_1(x_2))(s_2(a))\\
 && -\gamma_2 s_2(\rho(x_1, x_2)a)+s_2(\rho(x_1, x_2)a)-{\widehat{\rho}}(s_1(x_1), s_1(x_2))(s_2(a))\\
&=&{\widehat{\rho}}(s_1(x_1), s_1(x_2))\eta_{\gamma_2}(a)  +{\widehat{\rho}}(\zeta_{\gamma_1}(x_1), s_1(x_2))(s_2(a))-{\widehat{\rho}}(\zeta_{\gamma_1}(x_2), s_1(x_1))(s_2(a))\\
&&- {\widehat{\rho}}^l(\gamma_ 1 s_1(x_1),\gamma_ 1 s_1(x_2))(\gamma_2 s_2(a))+{\widehat{\rho}}(s_1(x_1), s_1(x_2))(s_2(a))+\gamma_2\nu(x_1,x_2)a-\nu(x_1,x_2)a\\
&=&-{\widehat{\rho}}(\zeta_{\gamma_1}(x_1),x_2)(\eta_{\gamma_2}(a)) +\gamma_2\nu(x_1,x_2)a-\nu(x_1,x_2)a\\
&=&0.
\end{eqnarray*}
By the same token, we can prove that $(\zeta_{\gamma_1},\eta_{\gamma_2})$ satisfy the other identities in $ \mathcal{Z}^1_\mathrm{MPL}(\mathfrak{g} \Join \mathfrak{h},V \Join W)$.
Thus, $\varPsi$ is well-defined.

Secondly, for any $\varUpsilon,\varUpsilon'\in \mathrm{Aut}_{V \Join W}^{\mathfrak{g} \Join \mathfrak{h}}(\widehat{\mathfrak{g}} \Join \widehat{\mathfrak{h}})$ and $x\in \mathfrak{g}$, suppose $\varPsi(\varUpsilon)=(\varPsi(\gamma_1),\varPsi(\gamma_2))=(\zeta_{\gamma_1},\eta_{\gamma_2})$
and $\varPsi(\varUpsilon')=(\varPsi(\gamma^{'}_{1}),\varPsi(\gamma^{'}_{2}))=(\zeta_{\gamma^{'}_{1}},\eta_{\gamma^{'}_2}).$
We get
\begin{align*}\varPsi(\gamma_1 \gamma^{'}_1)(x)&=\gamma_1 \gamma^{'}_{1}s_1(x)-s_1(x)
\\&=\gamma_1(\zeta_{\gamma'_1}(x)+s_1(x))-s_1(x)
\\&=\gamma_{1}\zeta_{\gamma'_1}(x)+\gamma_1 s_1(x)-s_1(x)
\\&=\zeta_{\gamma_1}(x)+\zeta_{\gamma'_1}(x) \quad (\because ~ \gamma_{1}|_V = \mathrm{id}_V)
\\&=\varPsi(\gamma_1)(x)+\varPsi(\gamma'_1)(x).\end{align*}
Take the same procedure, one can verify that
\begin{equation*}\varPsi(\gamma_2 \gamma^{'}_2)(a)=\varPsi(\gamma_2)(a)+\varPsi(\gamma'_2)(a).\end{equation*}
Thus, $\varPsi(\varUpsilon \varUpsilon')=\varPsi(\varUpsilon)+\varPsi(\varUpsilon')$, that is, $\varPsi$ is a homomorphism of groups.

Finally, we prove that $\varPsi$ is bijective.  For all $\varUpsilon=(\gamma_1,\gamma_2)\in \mathrm{Aut}_{V \Join W}^{\mathfrak{g} \Join \mathfrak{h}}(\widehat{\mathfrak{g}} \Join \widehat{\mathfrak{h}})$, if $\varPsi(\gamma_1)=\zeta_{\gamma_1}=0$, we can get $\zeta_{\gamma_1}(x)=\gamma_1 s_1(x)-s_1(a)=0$, Therefore, for any $\widehat{x}=u+ s_1(x) \in \widehat{\mathfrak{g}}$, we have
$$ \gamma_1(\widehat{x}) = \gamma_1(u+ s_1(x))=\gamma_1(u)+\gamma_1 s_1(x)-s_1(x)+s_1(x)=u+s_1(x)=\widehat{x},$$
that is, $\gamma_1=\mathrm{id}_{\hat{\mathfrak{g}}}$.
Similarly, $\gamma_2=\mathrm{id}_{\hat{\mathfrak{h}}}$. Thus, $\varPsi$ is injective. To show that the map $\varPsi$ is surjective,
 for any $(\zeta,\eta)\in \mathcal{Z}^1_\mathrm{MPL}(\mathfrak{g} \Join \mathfrak{h},V \Join W)$, define linear maps $\gamma_1:\widehat{\mathfrak{g}} \rightarrow \widehat{\mathfrak{g}} $
 and $\gamma_2:\widehat{\mathfrak{h}} \rightarrow \widehat{\mathfrak{h}}$ respectively by
  \begin{equation}\label{W7}\gamma_1(\widehat{x})=\gamma_1(u+s_1(x))=s_1(x)+\zeta(a)+u,~\forall~\widehat{x}\in \widehat{\mathfrak{g}},\end{equation}
  \begin{equation}\label{W8}\gamma_2(\widehat{a})=\gamma_1(w+ s_2(a))= s_2(a)+\eta(x)+w,~\forall~\widehat{a}\in \widehat{\mathfrak{h}}.\end{equation}
We need to verify that $\varUpsilon=(\gamma_1,\gamma_2) $ is an automorphism of the matched pair of 3-Lie algebras $\widehat{\mathfrak{g}} \Join \widehat{\mathfrak{h}}$.
It is obviously that $\gamma_1$ and $\gamma_2$ are bijective. Let $\widehat{x_1}, \widehat{x_2} ,\widehat{x_3}\in \widehat{\mathfrak{g}}$ and $\widehat{a_1}, \widehat{a_2}, \widehat{a_3} \in \widehat{\mathfrak{h}}$, it is not hard to see that $\gamma_1([\widehat{x_1}, \widehat{x_2}, \widehat{x_3}] )=[\gamma_1(\widehat{x_1}), \gamma_1(\widehat{x_2}), \gamma_1(\widehat{x_3})] $ and $\gamma_2([\widehat{a_1}, \widehat{a_2}, \widehat{a_3}] )=[\gamma_2(\widehat{a_1}), \gamma_2(\widehat{a_2}), \gamma_2(\widehat{a_3})] .$
Thus, we have $\gamma_1$ and $\gamma_2$ are isomorphism of the 3-Lie algebras.
By $\rho(x_1, x_2)\eta(x)-\eta(\rho(x_1, x_2)a)+\alpha(\zeta(x_1), x_2)a-\alpha(\zeta(x_2), x_1)a=0$, we have
\begin{eqnarray*}
&&\gamma_2 ({\widehat{\rho}}(\widehat{x_1}, \widehat{x_2})(\widehat{a})) \\
&=&\gamma_2 ({\widehat{\rho}}(u_1+s_1(x_1), u_2+s_1(x_2))(w+s_2(a))) \\
&=&\gamma_2 ({\widehat{\rho}}(s_1(x_1), s_1(x_2))( w)+{\widehat{\rho}}(u_1, s_1(x_2))(s_2(a))-{\widehat{\rho}}(u_2, s_1(x_1))(s_2(a))+{\widehat{\rho}}(s_1(x_1), s_1(x_2))s_2(a))\\
&=&\gamma_2 (\rho_W(x_1, x_2)w+\alpha(u_1,x_2)a-\alpha(u_2, x_1)a+\nu(x_1, x_2)a+\eta(\rho(x_1, x_2)a)+s_2(\rho(x_1, x_2)a))\\
&=&\rho_W(x_1, x_2)(w)+\alpha(u_1,x_2)(a)-\alpha(u_2,x_1)\alpha_(a)+\nu(x_1, x_2)(a)+\rho(x_1,x_2)\eta(a)+\alpha(\zeta(x_1),x_2)(a)\\
&&-\alpha(\zeta(x_2), x_1)(a)+s_2(\rho(x_1, x_2)\alpha(a))\\
&=& {{\widehat{\rho}}}(s_1(x_1),s_1(x_2))(w+\eta(a))+ {{\widehat{\rho}}}(u_1+\eta(x_1), s_1(x_2))s_2(a)\\
&&-{{\widehat{\rho}}}(u_2+\eta(x_2), s_1(x_1))s_2(a)+{{\widehat{\rho}}}(s_1(x_1),s_1(x_2))(s_2(a))\\
&=&{{\widehat{\rho}}}(s_1(x_1)+\zeta(x_1)+u_1 ,  s_1(x_2)+\zeta(x_2)+u_2)(s_2(a)+\eta(a)+w)\\
&=&{{\widehat{\rho}}}(\gamma_1(\widehat{x_1}), (\widehat{x_2}))(\gamma_2(\widehat{a})),
\end{eqnarray*}
which show that $\gamma_2 ({\widehat{\rho}}(\widehat{x_1}, \widehat{x_2})(\widehat{a})) = {{\widehat{\rho}}}(\gamma_1(\widehat{x_1}), (\widehat{x_2}))(\gamma_2(\widehat{a}))$. Similarly, we can get $\gamma_1 ({\widehat{\psi}}(\widehat{a_1}, \widehat{a_2})(\widehat{x})) = {{\widehat{\psi}}}(\gamma_2(\widehat{a_1}), (\widehat{a_2}))(\gamma_1(\widehat{x}))$.
 Thus, $\varUpsilon=(\gamma_1,\gamma_2)\in  \mathrm{Aut}(\widehat{\mathfrak{g}} \Join \widehat{\mathfrak{h}})$.
 Moreover, we have $(j_1\gamma_1 s_1,\gamma_{1}|_{V})=(\mathrm{id}_\mathfrak{g},\mathrm{id}_V)$ and
 $( j_2 \gamma_2 s_2,\gamma_{2}|_{W})=(\mathrm{id}_\mathfrak{h},\mathrm{id}_W)$. It follows that $\varUpsilon=(\gamma_1,\gamma_2) \in \mathrm{Aut}_{V \Join W}^{\mathfrak{g} \Join \mathfrak{h}}(\widehat{\mathfrak{g}} \Join \widehat{\mathfrak{h}})$. Thus, $\varPsi$ is surjective. In all, $\varPsi$ is bijective.
 So, $\mathcal{Z}^1_\mathrm{MPL}(\mathfrak{g} \Join \mathfrak{h},V \Join W) \cong \mathrm{Aut}_{V \Join W}^{\mathfrak{g} \Join \mathfrak{h}}(\widehat{\mathfrak{g}} \Join \widehat{\mathfrak{h}})$.
\end{proof}

\begin{thm} Let
  $\mathcal{E}:\xymatrix{
0 \ar[r] & V \Join W \ar[r]^{i_1 \Join i_2}  &  \widehat{\mathfrak{g}} \Join \widehat{\mathfrak{h}} \ar[r]^{j_1 \Join j_2} & \mathfrak{g} \Join \mathfrak{h} \ar[r] & 0
}$
be an abelian extension of $\mathfrak{g} \Join \mathfrak{h}$ by $V \Join W$. There is an exact sequence:
$$0\longrightarrow \mathcal{Z}^1_\mathrm{MPL}(\mathfrak{g} \Join \mathfrak{h},V \Join W)\stackrel{\iota}{\longrightarrow} \mathrm{Aut}_{V \Join W}(\widehat{\mathfrak{g}} \Join \widehat{\mathfrak{h}})
\stackrel{\varPhi}{\longrightarrow}\mathcal{C} \stackrel{\mathcal {W}}{\longrightarrow} \mathcal{H}^2_\mathrm{MPL}(\mathfrak{g} \Join \mathfrak{h},V \Join W).$$
\end{thm}
\begin{proof}Since the inclusion map $\iota : \mathcal{Z}^1_\mathrm{MPL} (\mathfrak{g} \Join \mathfrak{h},V \Join W) \to \mathrm{Aut}_{V \Join W}(\widehat{\mathfrak{g}} \Join \widehat{\mathfrak{h}})$ is an injection, the above sequence is exact at the first term.

Next, since $\mathrm{ker} (\varPhi)=\mathrm{Aut}_{V \Join W}^{\mathfrak{g} \Join \mathfrak{h}}(\widehat{\mathfrak{g}} \Join \widehat{\mathfrak{h}})$.
By Theorem \ref{Der}, we have $\mathrm{Aut}_{V \Join W}^{\mathfrak{g} \Join \mathfrak{h}}(\widehat{\mathfrak{g}} \Join \widehat{\mathfrak{h}})\cong \mathcal{Z}^1_\mathrm{MPL} (\mathfrak{g} \Join \mathfrak{h},V \Join W).$
Thus, we have $\mathrm{ker} (\varPhi) =\mathrm{im}(\iota)$. This shows that the sequence is exact at the second term.

Finally, to show that the sequence is exact at the third term, take a pair $(\alpha,\beta) \in \mathrm{ker}(\mathcal{W})$, that is $\mathcal {W}(\alpha,\beta)=0$. Thus, we have the abelian 2-cocycles $(\omega,\theta,\nu,\phi)_{(\alpha,\beta)}$ and $(\omega,\theta,\nu,\phi)$ are equivalent, by Theorem \ref{CIdu}, the pair $(\alpha,\beta)$ is inducible. In other words, there exists an automorphism $\varUpsilon \in \mathrm{Aut}_{V \Join W}(\widehat{\mathfrak{g}} \Join \widehat{\mathfrak{h}})$ such that $\varPhi(\varUpsilon) = (\alpha,\beta)$. This shows that $(\alpha,\beta) \in \mathrm{im}(\varPhi)$. Conversely, if a pair $(\alpha,\beta) \in \mathrm{im}(\varPhi)$, then by definition the pair $(\alpha,\beta)$ is inducible. Hence again by Theorem \ref{CIdu}, the abelian 2-cocycles $(\omega,\theta,\nu,\phi)_{(\alpha,\beta)}$ and $(\omega,\theta,\nu,\phi)$ are equivalent. Therefore, $\mathcal{W}(\alpha,\beta) = 0$ which implies that $(\alpha,\beta) \in \mathrm{ker}(\mathcal{W})$. Thus, we obtain $\mathrm{ker}(\mathcal{W}) = \mathrm{im}(\varPhi)$. This completes the proof.
\end{proof}

\end{document}